\documentclass[reqno]{amsart}

\usepackage{amsmath,amssymb,amsfonts,amsthm}
\usepackage{cancel}
\usepackage{pgf,tikz}
\usepackage{float}
\usetikzlibrary{arrows}
\usepackage{mathdots}
\usepackage{hyperref}

 \usepackage{color}

\newcommand{\dtau}{\ \mathrm{d}\tau}

\theoremstyle{plain}
\newtheorem{thm}{Theorem}[section]
\newtheorem{prop}[thm]{Proposition}
\newtheorem{cor}[thm]{Corollary}
\newtheorem{lem}[thm]{Lemma}

\theoremstyle{definition}
\newtheorem{defn}[thm]{Definition}

\newtheorem{ex}[thm]{Example}
\newtheorem{rem}[thm]{Remark}

\numberwithin{equation}{section}
\newcommand{\BC}{{\mathbb C}}

\newcommand{\BR}{{\mathbb R}}

\newcommand{\BZ}{{\mathbb Z}}
\newcommand{\cC}{{\mathcal C}}

\newcommand{\cG}{{\mathcal G}}
\newcommand{\cI}{{\mathcal I}}

\newcommand{\cN}{{\mathcal N}}
\newcommand{\cO}{{\mathcal O}}\newcommand{\cP}{{\mathcal P}}
\newcommand{\cQ}{{\mathcal Q}}\newcommand{\cR}{{\mathcal R}}
\newcommand{\cS}{{\mathcal S}}\newcommand{\cT}{{\mathcal T}}
\newcommand{\cU}{{\mathcal U}}\newcommand{\cV}{{\mathcal V}}
\newcommand{\cX}{{\mathcal X}}
\newcommand{\cY}{{\mathcal Y}}\newcommand{\cZ}{{\mathcal Z}}
\newcommand{\wtilG}{\widetilde{G}}

\newcommand{\la}{\lambda}

\newcommand{\si}{\sigma}\newcommand{\Si}{\Sigma}

\newcommand{\im}{\textup{Im}\,}
\newcommand{\re}{\textup{Re}}
\newcommand{\kr}{\textup{Ker\,}}
\newcommand{\diag}{\textup{diag\,}}

\newcommand{\mat}[1]{\ensuremath{\begin{bmatrix} #1 \end{bmatrix}}}
\newcommand{\matt}[2]{\ensuremath{\left[\begin{array}{#1}#2\end{array} \right]}}

\newcommand{\ov}[1]{{\overline{#1}}}
\newcommand{\un}[1]{{\underline{#1}}}

\newcommand{\wtil}[1]{{\widetilde{#1}}}
\newcommand{\what}[1]{{\widehat{#1}}}
\newcommand{\ands}{\quad\mbox{and}\quad}

\newcommand{\dhup}{\ensuremath{\rotatebox{90}{\!$\twoheadrightarrow$}}}
\newcommand{\dhdown}{\ensuremath{\rotatebox{90}{\!$\twoheadleftarrow$}}}

\begin{document}

%\title{Realization theory for poset-causal systems}

\title[Realization theory for poset-causal systems]{Realization theory for poset-causal systems:\ Controllability, observability and duality}

\author[S. ter Horst]{S. ter Horst}
\address{S. ter Horst, Department of Mathematics, Research Focus Area:\ Pure and Applied Analytics, North-West University, Potchefstroom, 2531 South Africa and DSI-NRF Centre of Excellence in Mathematical and Statistical Sciences (CoE-MaSS)}
\email{Sanne.TerHorst@nwu.ac.za}

\author[J. Zeelie]{J. Zeelie}
\address{J. Zeelie, Department of Mathematics,  Research Focus Area:\ Pure and Applied Analytics, North-West University, Potchefstroom 2531, South Africa}
\email{24698245@nwu.ac.za}

\thanks{This work is based on the research supported in part by the National Research Foundation of South Africa (Grant Numbers 118513 and 127364).}

%93A14 Decentralized systems
%93B05 Controllability
%93B07 Observability
%93C05 Linear systems


\maketitle

\begin{abstract}
Poset-causal systems form a class of decentralized systems introduced by Shah and Parrilo \cite{SP08} and studied mainly in the context of optimal decentralized control. In this paper we develop part of the realization theory for poset-causal systems. More specifically, we investigate several notions of controllability and observability, and their relation under duality. These new notions extend concepts of controllability and observability in the context of coordinated linear systems \cite{KRvS12}. While for coordinated linear systems there is a clear hierarchical structure with a single (main) coordinator, for poset-causal systems there need not be a single  coordinator and the communication structure between the decentralized systems allows for more intricate structures, governed by partial orders. On the other hand, we show that the class of poset-causal systems is closed under duality, which is not the case for coordinated linear systems, and that duality relations between the various notions of observability and controllability exist.
\end{abstract}

%%%%%%%%%%%%%%%%%%%%%%%%%%%%%%%%%%%%%%%%%%%%%%%%%%%%%%%%%%%%%%%%%%%%%%%%%%%%%%%%%%%%%%%%%%%%%%
\section{Introduction}\label{S:Intro}

In many practical applications, the systems under consideration are large scale and consist of spatially distributed, but interconnected subsystems. Information flow in such systems occur in a distributed manner and such systems lend themselves to decentralized control strategies. Examples include: large scale irrigation systems \cite{C08,LC08,NvODS09,SvOB14,W08}, hydroelectricity plants \cite{FMBF13,FMBF11} inland navigation networks \cite{RHDC14,SRND17}, civil building systems \cite{JKSA14,LL02,MBDB10} and electric power systems \cite{MDSLCBL17}.

Such interconnected systems are often approached using graph theoretic techniques with nodes representing subsystems and directed edges representing communication flow (see for example \cite{SL09,SL14}). An equivalent approach was introduced by Shah and Parrilo in \cite{SP08} using partial orders, leading to the notion of  poset-causal systems, which were further studied in  the papers \cite{SP09,SP11,SP13} and the PhD thesis \cite{S11}. An important sub-class of poset-causal systems are so-called coordinated linear systems, which were defined in \cite{RvS08}.

Much of the research on the controllability of decentralized systems take a graph-theoretic approach where leaders are to be chosen from multiple agents in a communication topology in such a way that renders the system controllable. In \cite{T04}, the controllability of an interconnected system by choosing a single leader among first-order subsystems, is considered. This was extended to multiple leaders in \cite{JE07}. A graphical characterization of controllability in this context was developed in \cite{JLY12}.  In the paper \cite{CZ10}, the concepts of formation controllability and complete controllability are studied for multi-agent swarm systems where all the agents are LTI systems and have the same order. The controllability of networked systems is studied in \cite{WCWT16}.

Several concepts of controllability and observability of coordinated linear systems with one leader and two followers were studied in \cite{KRvS12}. The aim was to develop concepts of controllability and observability that respect the communication structure of the system. In the current paper we study controllability for the larger class of poset-causal systems. Due to the more intricate structures of such systems, our notions are not based in leader-follower concepts, but rather use upstream and downstream subsystems from the perspective of the given subsystems. These lead to notions of controllability that respect the poset-causal structure, in the sense that the space of controllable state vectors decompose in direct sums of subspaces of the local state spaces corresponding to the subsystems of the poset-causal system. One notion, independently controllable, is stronger than classical controllability, in the sense that it implies classical controllability, and one notion, weak local controllability, is weaker, in the sense that it is implied by classical controllability. In fact, in a way (see Theorem \ref{T:Contr}) weak local controllability is the strongest notion of controllability implied by controllability that preserves the poset-causal structure. Similar notions exist for observability for poset-causal systems, and it turns out that they are related through duality, as in the classical setting. Finally, we investigate how our notions of controllability and observability can be used to perform a Kalman type reduction of the system in a way that preserves the poset-causal structure.

Poset-causal systems are decentralized systems that consist of interconnected subsystems labeled $1,2,\ldots,p$. The interconnections are modeled by a partial order $\succeq$ on the set $P=\{1,2,\ldots,p\}$. Hence, the pair $\cP=(P,\succeq)$ is a partially ordered set, or poset for short. In this setting subsystem $j$ can `influence' subsystem $i$ in case $j \succeq i$. For reasons of notational convenience, the order in the current paper is reversed compared to the order used in the work of Shah and Parrilo. We assume that each subsystem is locally given by an input-state-output model with an input space $\cU_i=\BR^{m_i}$, a state space $\cX_i=\BR^{n_i}$ and an output space $\cY_i=\BR^{r_i}$, with $m_i,n_i,r_i\in\BZ_+$. The local outputs $y_i(t)\in\cY_i$ and local states $x_i(t)\in\cX_i$ are determined by the states $x_j(t)\in\cX_j$ and inputs $u_j(t)\in\cU_j$ of subsystems $j$ that can `influence' subsystem $i$ via interconnected state space system equations
\begin{equation}\label{eqSubsys}
\begin{aligned}
\dot{x}_i(t) & = \sum_{j\in{\uparrow}i} A_{ij}x_j(t) + \sum_{j\in{\uparrow}i} B_{ij}u_j(t), \quad x_i(0)=x_{i,0},\\
y_i(t) & = \sum_{j\in{\uparrow}i} C_{ij}x_j(t) + \sum_{j\in{\uparrow}i} D_{ij}u_j(t),\quad t\geq 0,
\end{aligned}
\end{equation}
where ${\uparrow}i=\{j\mid j\succeq i\}$ is the set of subsystems $j$ that are upstream of subsystem $i$ in the communication network. That is, if $j\in{\uparrow}i$, then subsystem $j$ can `influence' subsystem $i$. Here $x_{i,0}\in\BR^{n_i}$ is the initial state of subsystem $i$ and $A_{ij}\in\BR^{n_i\times n_j}$, $B_{ij}\in\BR^{n_i\times m_j}$, $C_{ij}\in\BR^{r_i\times n_j}$ and $D_{ij}\in\BR^{r_i\times m_j}$ are given matrices whenever $j\succeq i$. In case $j\not\succeq i$, set $A_{ij}$, $B_{ij}$, $C_{ij}$ and $D_{ij}$ equal to zero matrices of appropriate sizes and define
\begin{equation}\label{BlockPart}
A=[A_{ij}]_{i,j=1}^p,\quad B=[B_{ij}]_{i,j=1}^p,\quad C=[C_{ij}]_{i,j=1}^p,\quad D=[D_{ij}]_{i,j=1}^p.
\end{equation}
Then the combined input, state and output signals
\begin{align*}
&u(t)=(u_1(t),\ldots,u_p(t))^\intercal\in \cU:=\bigoplus_{i=1}^p\cU_i,\quad
x(t)=(x_1(t),\ldots,x_p(t))^\intercal\in \cX:=\bigoplus_{i=1}^p\cX_i,\\
&\qquad\qquad\qquad\qquad\qquad y(t)=(y_1(t),\ldots,y_p(t))^\intercal\in \cY:=\bigoplus_{i=1}^p\cY_i
\end{align*}
%\begin{align*}
%u(t)&=(u_1(t),\ldots,u_p(t))\in \cU=\bigoplus_{i=1}^p\cU_i,\\
%x(t)&=(x_1(t),\ldots,x_p(t))\in \cX=\bigoplus_{i=1}^p\cX_i,\\
%y(t)&=(y_1(t),\ldots,y_p(t))\in \cY=\bigoplus_{i=1}^p\cY_i
%\end{align*}
satisfy
\begin{equation}\label{eqLTI}
\begin{aligned}
\dot{x}(t) & = Ax(t)+Bu(t), \quad x(0)=x_0\in\cX,\\
y(t)   & = Cx(t) +Du(t),\quad t\geq 0.
\end{aligned}
\end{equation}
where $x_0:=(x_{1,0},\ldots,x_{p,0})^\intercal$. Hence the decentralized system \eqref{eqSubsys} can be written as a classical state space system \eqref{eqLTI} with the communication structure embedded in a prescribed block zero pattern of the system matrices determined by the underlying partial order. In particular, the state and output signals can be represented in terms of the input and initial state by the classical integral formulas
\begin{equation}\label{eqintegralform}
\begin{aligned}
x(t)&=x(x_0,u,t) = e^{At}x_0+\int_{0}^{t}e^{A(t-\tau)}Bu(\tau) \dtau,\\
y(t)&=y(x_0,u,t) = Ce^{At}x_0+\int_0^tCe^{A(t-\tau)}Bu(t)\dtau + Du(t).
\end{aligned}
\end{equation}
Since $\succeq$ is a partial order, the block zero-pattern in the system matrices \eqref{BlockPart} is invariant under block matrix multiplication (provided the block sizes are compatible for multiplication) and under matrix inversion. As a consequence, the frequency domain transfer function $F(s)=D+C(sI-A)^{-1}B$ has the same block zero-pattern as the system matrices.  Conversely, it is shown in \cite{LKR13} that if $F$ is a rational matrix function which does not have a pole at 0 and of which the values have a block zero-pattern prescribed by a partial order $\succeq$ (and even less restrictive structures), then $F$ is the transfer function of a poset-causal system associated with $\succeq$. However, it is also shown in \cite{LKR13} that it might not be possible to construct a poset-causal system whose transfer function is $F$ in such a way that the poset-causal system is stabilizable and detectable. Similar issues occur with controllability and observability.

Since the poset-causal system \eqref{eqSubsys} can be represented as a classical state space system \eqref{eqLTI} with structured system matrices, all the notions, results and constructions from classical state space theory apply. However, most of these do not preserve the block zero-pattern. For example, the \emph{reachable subspace} $\cR(A,B)$ and \emph{unobservable subspace} $\cN(C,A)$ of the state space $\cX$ given by
\begin{equation}\label{eqreachunobs}
\cR(A,B):=\im \cC(A,B) \ands \cN(C,A):=\kr \cO(C,A),
\end{equation}
with $\cC(A,B)$ and $\cO(C,A)$ the {\em controllability} and {\em observability matrices}, respectively, associated with \eqref{eqLTI}:
\begin{equation}\label{eqcCcO}
\cC(A,B) := \begin{bmatrix}
B & AB & A^2B & \ldots & A^{n-1}B
\end{bmatrix} \ands
\cO(C,A):= \cC(A^\intercal,C^\intercal)^\intercal,
\end{equation}
where $n=\sum_{j\in P} n_j$ is the state space dimension, cannot, in general, be written as direct sums of subspaces of the local state spaces $\cX_j$ for $j\in P$, so that the compression to a minimal system obtained from the Kalman decomposition will in general not have the appropriate zero-pattern, destroying the poset-causal structure.

As a result, it is not clear how to determine if a poset-causal system is minimal, that is, whether there does not exist an poset-causal system which generates the same input-output map but has smaller state space dimensions. Neither is it clear whether there exists a Kalman-type decomposition if it is also required that the block zero-pattern be preserved. New notions of controllability and observability are required that do preserve the communication structure and still preserve some of the features of the classical notions.

For the subclass of coordinated linear systems defined in \cite{RvS08} various new concepts of controllability and observability that preserve the underlying communication structure  were studied in \cite{KRvS12}, see also \cite{KRvS14b}. Coordinated linear systems are poset-causal systems for which the partial order satisfies a stronger form of transitivity defined in \cite{BES16}, namely in-ultra transitivity, see Definition \ref{D:ultatransitive} below. However, they also form a subclass of hierarchical systems \cite{FBBMTW80}, with a clear hierarchical structure, namely, the transitive closures of directed trees, and a single (main) coordinator. Most of the controllability and observability notions in \cite{KRvS12} are defined for a coordinated linear systems that consists of a single coordinator and two followers. Here, the coordinator can communicate with the followers, but the followers cannot communicte with each other or with the coordinator.

In the present paper we introduce various notions of controllability and observability for poset-causal systems related to upstream and downstream systems associated with the subsystems. When specialized to the setting of coordinated linear systems, these reduce to notions of controllability and observability studied in \cite{KRvS12}, however, the notions introduced here do not rely on the roles of ``coordinator'' or ``follower'' a subsystem may have. While it is possible to study controllability and observability for systems defined on graphs from the perspective of assigning controllers, see e.g.\ \cite{MZC14}, in this paper we do not try to assign specific roles to the subsystems. The class of coordinated linear systems is not closed under duality. However, it turns out that this is the case for the class of poset-causal systems, and we prove duality relations between the controllability and observability notions defined in this paper, as it occurs in the classical case.

Most of the notions of controllability and observability introduced here are based on variations on the classical reachable subspace and unobservable subspace, with the difference that they can be written as (orthogonal) direct sums of subspaces of the local state spaces of the subsystems. As a result of this, we are able to present a variation on the Kalman reduction formula, in Section \ref{S:Kalman}. Despite that some of the controllability and observability notions introduced here are optimal (in the sense presented in Theorems 3.3 and 4.3), this does not carry over to the Kalman reduction obtained here. Hence, although it may compress the poset-causal system to a poset-causal system with much smaller state space dimesions, it need not be the minimal poset-causal system with  the same input-output map.

We conclude this introduction with an overview of the paper. In Section \ref{S:PosetCausal} we give a more precise definition of poset-causal systems as well as various related subsystems and we present some of the preliminaries used in the remainder of the paper. Section \ref{S:UpControl} contains various notions of controllability based on the concept of downstream reachable states and we prove their relation to classical controllability. This is followed by a similar discussion of notions related to observability in Section \ref{S:DownObserve}. The duality relations between the new notions of controllability and observability are proved in Section \ref{S:Dual}. Finally, in Section \ref{S:Kalman}, we consider the problems related to the minimality of poset-causal and we employ a Kalman-type reduction of poset-causal systems.

%\newpage
%%%%%%%%%%%%%%%%%%%%%%%%%%%%%%%%%%%%%%%%%%%%%%%%%%%%%%%%%%%%%%%%%%%%%%%%%%%%%%%%%%%%%%%%%%%%%%%%
\section{Poset-causal systems}\label{S:PosetCausal}

In this section we give a formal definition of poset-causal systems and introduce the dual of a poset-causal system. For this we require some preliminary definitions and results on order structures and matrices with associated block zero-patterns.

%=================================================================================================================================
\subsection{Order structures}
A {\em partially ordered set}, or {\em poset}, is a pair $\cP=(P,\succeq)$ with $P$ a set and $\succeq$ a partial order on $P$. That is, $\succeq$ is a binary relation on $P$ which is
\begin{itemize}
	\item[(i)] {\em reflexive}:\ $i\succeq i$ for all $i\in P$;
	\item[(ii)] {\em transitive}:\ if $i\succeq j$ and $j\succeq k$, then $i\succeq k$ for all $i,j,k\in P$;
	\item[(iii)] {\em anti-symmetric}:\ if $i\succeq j$ and $j\succeq i$, then $i=j$ for all $i,j\in P$.
\end{itemize}
If a binary relation only satisfies (i) and (ii), then it is a \emph{pre-order}. For $i,j\in P$ we write $i\succ j$ if $i\succeq j$ and $i\neq j$. Also, $i\preceq j$ and $i\prec j$ means $j\succeq i$ and $j\succ i$, respectively. In the sequel we will only consider finite posets, usually with $P=\{1,2,\ldots,p\}$ for some positive integer $p$. Given a subset $R\subseteq P$ of a poset $\cP=(P,\succeq)$ we define its {\em downstream set} ${\downarrow}R$ and its {\em upstream set} ${\uparrow}R$ as
\begin{equation}\label{UpDownOff}
	\begin{aligned}
		&{\downarrow} R :=\{i\in P \colon j\succeq i \mbox{ for some }j\in R\},\\
		& {\uparrow} R :=\{i\in P \colon i\succeq j \mbox{ for some }j\in R\}.
	\end{aligned}
\end{equation}
In the case that $R$ is a singleton, say $R=\{i\}$, we simply write ${\downarrow} i$ and ${\uparrow} i$. By reflexivity $R\subseteq {\uparrow}R$ and $R\subseteq {\downarrow}R$. By transitivity ${\uparrow}({\uparrow}R) \subseteq {\uparrow}R$ and ${\downarrow}({\downarrow}R) \subseteq {\downarrow}R$. Together with reflexivity, this gives ${\uparrow}({\uparrow}R) = {\uparrow}R$ and ${\downarrow}({\downarrow}R) = {\downarrow}R$. In addition, we define
\[
{\dhdown} R:=\{i\in {\downarrow}R \colon i\not\in R\} \ands
{\dhup} R:=\{i\in {\uparrow}R \colon i\not\in R\},
\]
again abbreviated to ${\dhdown}i$ and ${\dhup}i$, respectively, when $R=\{i\}$.

Any poset $\cP=(P,\succeq)$ can be represented by a digraph $\cG_\cP=(P,E_{\succeq})$ where the nodes in $\cG_\cP$ are the elements of $P$ and $E_{\succeq}=\{(i,j)\in P\times P: i\succeq j\}$ is the set of directed edges in $\cG_\cP$. The Hasse diagram of a poset $\cP=(P,\succeq)$ can be identified with the digraph $\cG^{\downarrow}_\cP=(P,E^{\downarrow}_{\succeq})$, where
\[
E^{\downarrow}_{\succeq} := \{(i,j)\in P\times P: i\succ j\text{ and  there is no $k\in P$ such that $i\succ k \succ j$} \}.
\]
The digraph $\cG_\cP^{\downarrow}$ omits all the directed edges that correspond to reflexivity and transitivity.

\begin{ex}
Consider the poset $\cP=(P,\succeq)$ with $P=\{1,2,3,4\}$ determined by $1\succeq2$, $3\succeq2$ and $2\succeq4$ (along with loops and edges induced by transitivity). The digraph $\cG_\cP=(P,E_\succeq)$ and Hasse diagram $\cG_\cP^{\downarrow}=(P,E_{\succeq}^{\downarrow})$ of $\cP$ are given by:
%
%The digraph $\cG_\cP=(P,E_\cP^{\downarrow})$ has vertices $1,2,3,4$ and directed edges $(i,j)$ if $i\succeq j$ for $i,j\in P$. Hence the directed edges $(1,2)$, $(3,2)$ and $(2,4)$ are in $\cG_\cP$. Since $\succeq$ is reflexive, each vertex has a loop. Since $\succeq$ is transitive, the directed edges $(1,4)$ and $(3,4)$ are also in $\cG_\cP$. The Hasse diagram $\cG_\cP^{\downarrow}$ of $\cP$ does not contain the directed edges corresponding to reflexivity and transitivity:
	
	
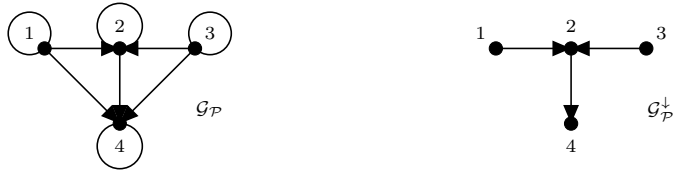
\begin{tikzpicture}[line cap=round,line join=round,>=triangle 45,x=1.0cm,y=1.0cm]
	\clip(-1,0) rectangle (10,3.1159485608468165);
	\draw [->,line width=0.6pt] (0.5,2.) -- (1.5,2.);
	\draw [->,line width=0.6pt] (2.5,2.) -- (1.5,2.);
	\draw [->,line width=0.6pt] (1.5,2.) -- (1.5,1.);
	\draw [->,line width=0.6pt] (0.5,2.) -- (1.5,1.);
	\draw [->,line width=0.6pt] (2.5,2.) -- (1.5,1.);
	\draw [line width=0.6pt] (0.3,2.2) circle (0.2828427124746192cm);
	\draw [line width=0.6pt] (1.5,2.3) circle (0.3cm);
	\draw [line width=0.6pt] (2.7,2.2) circle (0.2828427124746193cm);
	\draw [line width=0.6pt] (1.5,0.7) circle (0.3cm);
	\begin{scriptsize}
	\draw [fill=black] (0.5,2.) circle (2.5pt);
	\draw[color=black] (0.3,2.2) node {$1$};
	\draw [fill=black] (1.5,2.) circle (2.5pt);
	\draw[color=black] (1.5,2.3) node {$2$};
	\draw [fill=black] (2.5,2.) circle (2.5pt);
	\draw[color=black] (2.7,2.2) node {$3$};
	\draw [fill=black] (1.5,1.) circle (2.5pt);
	\draw[color=black] (1.5,0.7) node {$4$};
	\draw[color=black] (2.7,1.2) node {$\cG_\cP$};
	\end{scriptsize}

	\draw [->,line width=0.6pt] (6.5,2.) -- (7.5,2.);
	\draw [->,line width=0.6pt] (8.5,2.) -- (7.5,2.);
	\draw [->,line width=0.6pt] (7.5,2.) -- (7.5,1.);
	\begin{scriptsize}
	\draw [fill=black] (6.5,2.) circle (2.5pt);
	\draw[color=black] (6.3,2.2) node {$1$};
	\draw [fill=black] (7.5,2.) circle (2.5pt);
	\draw[color=black] (7.5,2.3) node {$2$};
	\draw [fill=black] (8.5,2.) circle (2.5pt);
	\draw[color=black] (8.7,2.2) node {$3$};
	\draw [fill=black] (7.5,1.) circle (2.5pt);
	\draw[color=black] (7.5,0.7) node {$4$};
	\draw[color=black] (8.7,1.2) node {$\cG^{\downarrow}_\cP$};
	\end{scriptsize}
	\end{tikzpicture}
	
	We also illustrate some upstream and downstream sets:
	\begin{align*}
	{\downarrow}1 = \{1,2,4\}, \quad {\dhdown}1 = \{2,4\}, \quad {\uparrow}4 = \{1,2,3,4\}=P \quad \ands \quad {\dhup}2=\{1,3\}.
	\end{align*}
\end{ex}

In \cite{BES16,BES17,BES18} stronger notions of transitivity are studied, which are defined as follows.

\begin{defn}\label{D:ultatransitive}
	A binary relation $\cT=(P,T)$ is said to be
	\begin{itemize}
		\item[(i)]  an \emph{in-ultra transitive} relation if for all $i,j,k \in P$, $(i,j)\in T$ and $(k,j)\in T$ implies that $(i,k)\in T$ or $(k,i)\in T$, and
		\item[(ii)] an \emph{out-ultra transitive} relation if for all $i,j,k \in P$, $(i,j)\in T$ and $(i,k)\in T$ implies that $(j,k)\in T$ or $(k,j)\in T$.
	\end{itemize}
\end{defn}

We point out here that the Hasse diagrams of the underlying poset of coordinated linear systems as studied in \cite{KRvS12,KRvS14a,KRvS14b}, in the most general setting, are so-called out-tree forests, that is, a collection of directed trees that each have a single node (global coordinator) that all other nodes are directed away from. By Theorem 4.1 in \cite{BES16} out-tree forests correspond to posets which are in-ultra transitive.

\begin{defn}\label{D:dualposet}
	For a poset $\mathcal{P} = (P,\succeq)$, we define its \emph{dual poset} as $\mathcal{P}_d=(P,\succeq_d)$ where
	$$
	j\succeq_di \iff i\succeq j
	$$
	for each $j,i\in P$.
\end{defn}

In the sequel we will use the symbols ${\downarrow}_d$ and ${\uparrow}_d$ to indicate downstream and upstream sets, respectively, associated with the dual poset $\cP_d$. Clearly we have ${\uparrow}_d R = {\downarrow}R$ and ${\downarrow}_d R = {\uparrow}R$, for any $R\subseteq P$. The following lemma follows directly from the definitions.

\begin{lem}\label{L:in/out-ultra}
	The dual of an in-ultra transitive relation is and out-ultra transitive relation.
\end{lem}

\begin{ex}\label{ex5posets}
Consider posets $\cP_1$, $\cP_2$, $\cP_3$, $\cP_4$, $\cP_5$ and $\cP_6$ with Hasse diagrams given by:
	
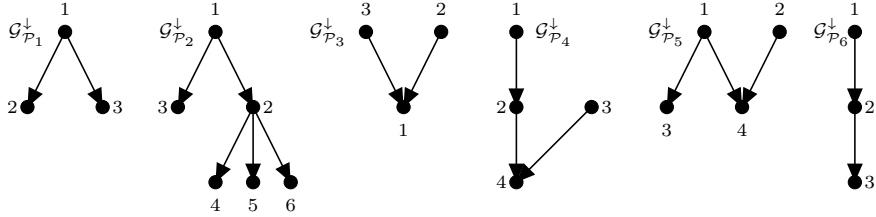
\begin{figure}[H]
		\centering
		\begin{tikzpicture}[line cap=round,line join=round,>=triangle 45,x=1.0cm,y=1.0cm]
		\clip(-0.5,1.5) rectangle (18,4.5);
		\draw [->,line width=0.6pt] (1.,4.) -- (0.5,3.);
		\draw [->,line width=0.6pt] (1.,4.) -- (1.5,3.);
		\draw [->,line width=0.6pt] (3.,4.) -- (2.5,3.);
		\draw [->,line width=0.6pt] (3.,4.) -- (3.5,3.);
		\draw [->,line width=0.6pt] (3.5,3.) -- (3.,2.);
		\draw [->,line width=0.6pt] (3.5,3.) -- (3.5,2.);
		\draw [->,line width=0.6pt] (3.5,3.) -- (4.,2.);
		\draw [<-,line width=0.6pt] (5.5,3.) -- (5.,4.);
		\draw [<-,line width=0.6pt] (5.5,3.) -- (6.,4.);
		\draw [->,line width=0.6pt] (7,4.) -- (7.,3.);
		\draw [->,line width=0.6pt] (7,3.) -- (7.,2.);
		\draw [->,line width=0.6pt] (8,3.) -- (7.,2.);
		\draw [->,line width=0.6pt] (9.5,4.) -- (9.,3.);
		\draw [->,line width=0.6pt] (9.5,4.) -- (10.,3.);
		\draw [->,line width=0.6pt] (10.5,4.) -- (10.,3.);
		\draw [->,line width=0.6pt] (11.5,4.) -- (11.5,3.);
		\draw [->,line width=0.6pt] (11.5,3.) -- (11.5,2.);
		
		\begin{scriptsize}
		\draw [fill=black] (1.,4.) circle (2.5pt);
		\draw[color=black] (1,4.3) node {$1$};
		\draw [fill=black] (0.5,3.) circle (2.5pt);
		\draw[color=black] (0.3,3) node {$2$};
		\draw [fill=black] (1.5,3.) circle (2.5pt);
		\draw[color=black] (1.7,3) node {$3$};
		\draw[color=black] (0.5,4) node {$\cG^{\downarrow}_{\cP_1}$};
		
		\draw [fill=black] (3.,4.) circle (2.5pt);
		\draw[color=black] (3,4.3) node {$1$};
		\draw [fill=black] (2.5,3.) circle (2.5pt);
		\draw[color=black] (2.3,3) node {$3$};
		\draw [fill=black] (3.5,3.) circle (2.5pt);
		\draw[color=black] (3.7,3) node {$2$};
		\draw [fill=black] (3.,2.) circle (2.5pt);
		\draw[color=black] (3,1.7) node {$4$};
		\draw [fill=black] (3.5,2.) circle (2.5pt);
		\draw[color=black] (3.5,1.7) node {$5$};
		\draw [fill=black] (4.,2.) circle (2.5pt);
		\draw[color=black] (4,1.7) node {$6$};
		\draw[color=black] (2.5,4) node {$\cG_{\cP_2}^{\downarrow}$};
		
		\draw [fill=black] (5.,4.) circle (2.5pt);
		\draw[color=black] (5,4.3) node {$3$};
		\draw [fill=black] (6.,4.) circle (2.5pt);
		\draw[color=black] (6,4.3) node {$2$};
		\draw [fill=black] (5.5,3.) circle (2.5pt);
		\draw[color=black] (5.5,2.7) node {$1$};
		\draw[color=black] (4.5,4) node {$\cG_{\cP_3}^{\downarrow}$};
		
		\draw [fill=black] (7,4.) circle (2.5pt);
		\draw[color=black] (7,4.3) node {$1$};
		\draw [fill=black] (7.,3.) circle (2.5pt);
		\draw[color=black] (6.8,3) node {$2$};
		\draw [fill=black] (7,2.) circle (2.5pt);
		\draw[color=black] (6.8,2) node {$4$};
		\draw [fill=black] (8.,3.) circle (2.5pt);
		\draw[color=black] (8.2,3) node {$3$};
		\draw[color=black] (7.5,4) node {$\cG_{\cP_4}^{\downarrow}$};
		
		\draw [fill=black] (9.5,4.) circle (2.5pt);
		\draw[color=black] (9.5,4.3) node {$1$};
		\draw [fill=black] (9.,3.) circle (2.5pt);
		\draw[color=black] (9,2.7) node {$3$};
		\draw [fill=black] (10.,3.) circle (2.5pt);
		\draw[color=black] (10,2.7) node {$4$};
		\draw [fill=black] (10.5,4.) circle (2.5pt);
		\draw[color=black] (10.5,4.3) node {$2$};
		\draw[color=black] (9,4) node {$\cG_{\cP_5}^{\downarrow}$};
		
		\draw [fill=black] (11.5,4.) circle (2.5pt);
		\draw[color=black] (11.5,4.3) node {$1$};
		\draw [fill=black] (11.5,3) circle (2.5pt);
		\draw[color=black] (11.7,3) node {$2$};
		\draw [fill=black] (11.5,2) circle (2.5pt);
		\draw[color=black] (11.7,2) node {$3$};
		\draw[color=black] (11.2,4) node {$\cG_{\cP_6}^{\downarrow}$};
		
		\end{scriptsize}
		\end{tikzpicture}
	\end{figure}

\noindent	
Then $\cG^{\downarrow}_{\cP_1}$ is the underlying digraph of a coordinated linear system with one coordinator and two followers,  $\cG_{\cP_2}^{\downarrow}$ is an out-tree corresponding to a more intricate coordinated linear system, $\cG_{\cP_3}^{\downarrow}$ is the dual of $\cG_{\cP_1}^\downarrow$ and hence corresponds to an out-ultra transitive ordering (in-tree), $\cG_{\cP_4}^\downarrow$ is also an in-tree, $\cG_{\cP_5}^\downarrow$ is the Hasse diagram of a poset which is neither in-ultra transitive nor out-ultra transitive and lastly $\cG_{\cP_6}$ corresponds to a complete order.
\end{ex}

%==============================================================================================================================
\subsection{Block matrices with prescribed zero-patterns}
%In this paper we are interested in classes of block matrices with prescribed (block) zero-patterns, which are not necessarily square, but are closed under (block) matrix multiplication, provided the sizes of the blocks are compatible. This requires us to introduce some notation.

Given some $n\in\BZ_+$, we will say $\un{n} = (n_1,n_2,\ldots,n_p)\in\BZ_+^p$ is a {\em partition} of $n$ if $|\un{n}|:=n_1+n_2+\ldots+n_p = n$. Let $\un{n}=(n_1,n_2,\ldots,n_p)\in\BZ_+^{p}$ and $\un{m}=(m_1,m_2,\ldots,m_q)\in\BZ_+^{q}$ be two given partitions. We will write $G=[G_{ij}]\in\BR^{\un{n}\times \un{m}}$, in which case it is to be understood that $G_{ij}\in\BR^{n_i \times m_j}$. If $\un{r}\in\BZ_+^p$ is another partition, then matrices $G\in\BR^{\un{n}\times \un{m}}$ and $H\in \BR^{\un{r}\times \un{n}}$ are said to be {\em compatible for the block matrix multiplication} $HG$.

Given a binary relation $\cT=(P,T)$ with $P=\{1,\ldots, p\}$ for some positive integer $p$, define the set of matrices
\[
\cI_\cT:=\{G=[g_{ij}]\in\BR^{p\times p} \colon \mbox{$g_{ij}=0$ if $\{k\in P\colon (j,k),(k,i)\in T\}=\emptyset$}\}.
\]
By the theorem on page 258 of \cite{D70} and the subsequent remark on page 259 it follows that the set $\cI_\cT$ forms a subalgebra of $\BR^{p\times p}$ if and only if the relation $\cT$ is transitive. It is then referred to as the {\em incidence algebra} associated with $\cT$.  If $\cP=(P,\succeq)$ is a poset, then $\cI_\cP$ is a unital matrix algebra which can also be written as
\[
\cI_\cP=\{G=[g_{ij}]\in\BR^{p\times p} \colon \mbox{$g_{ij}=0$ if $j\not\succeq i$}\}.
\]
%SHOULD WE MENTION HERE THAT WE FOLLOW SHAH AND PARRILO IN SETTING $g_{ij}=0$ IF $j\not\succeq i$ RATHER $g_{ij}=0$ IF $i\not\succeq j$ AS IN STANDARD DEFINITIONS?\footnote{Let's discuss this at the next meeting.}
By analogy of the incidence algebras defined above, we define block matrices with zero-pattern prescribed by a partial order.

\begin{defn}\label{D:BlockIncSet}
	Given a poset $\mathcal{P}=(P,\succeq)$, with $P=\{1,\ldots,p\}$ and partitions $\underline{n},\underline{m}\in\BZ_+^p$, we define the {\em block incidence vector space} $\mathcal{I}^{\underline{n}\times \underline{m}}_{\mathcal{P}}\subseteq \BR^{\un{n}\times \un{m}}$ as the subspace
	\begin{align*}
		\mathcal{I}^{\underline{n}\times \underline{m}}_{\cP} := \{G=[G_{ij}]\in\BR^{\underline{n}\times \underline{m}} \colon G_{ij}=0 \text{ if $j\not\succeq i$} \}.
	\end{align*}
\end{defn}

The fact that the set $P=\{1,2,\ldots,p\}$ is also ordered is not relevant in this paper, the choice to indicate $P$ in this way is just to clarify the relation to the columns and rows of the block matrices. Furthermore, since we only consider finite sets, $P$ can always be take in this form.

\begin{ex}\label{exMatrixZerostructures}
	Consider the posets $\cP_3$ and $\cP_5$ given in Example \ref{ex5posets} and partitions $\un{n},\un{m}\in\BZ_+^p$. The matrices $G$ and $H$ given below exhibit the block zero structures of matrices in the incidences spaces $\mathcal{I}^{\underline{n}\times \underline{m}}_{\cP_3}$ and $\mathcal{I}^{\underline{n}\times \underline{m}}_{\cP_5}$  respectively:
	\begin{align*}
		G=\begin{bmatrix}
			 G_{11} & G_{12} & G_{13}  \\
			 0 & G_{22} & 0  \\
			 0 & 0 & G_{33}
 	\end{bmatrix}
\quad \ands \quad
	H=\begin{bmatrix}
	 H_{11} & 0 & 0 & 0 \\
	0 & H_{22} & 0 & 0 \\
	 H_{31} & 0 & H_{33} & 0 \\
	 H_{41} & H_{42} & 0 & H_{44}
 	\end{bmatrix}.
	\end{align*}
\end{ex}

The numbering of nodes of a poset $\cP$ can always be done in such a way that the matrices in the corresponding incidence spaces are block lower triangular (for $\cP_3$ in Example \ref{exMatrixZerostructures} reorder $(1,2,3)\mapsto (3,2,1)$).

By arguments similar to those in \cite{D70} it follows that the block zero structure is preserved under block matrix multiplication, provided the block matrices are compatible for block matrix multiplication. The block zero structure is also invariant under inversion (since the inverse of an invertible matrix $A$ is contained in its double commutant $\{A\}''$).

\begin{prop}\label{P:MultClosed}
	Let $\mathcal{P}=(P,\succeq)$ be a poset with $p$ elements and let $\underline{n},\underline{m},\un{r}\in\BZ_+^p$. If $G\in \mathcal{I}^{\underline{r}\times \underline{n}}_{\cP}$ and $H\in \mathcal{I}^{\underline{n}\times \underline{m}}_{\cP}$, then $G H\in \mathcal{I}^{\underline{r}\times \underline{m}}_{\cP}$. If $G\in \mathcal{I}^{\underline{n}\times \underline{n}}_{\cP}$ and $\det G\neq 0$, then $G^{-1}\in \mathcal{I}^{\underline{n}\times \underline{n}}_{\cP}$.
\end{prop}

Throughout the paper we work with block compressions associated with subsets of $P$. Note that we have defined partitions in such a way that zero entries are permitted. It will be convenient in this paper to define block compressions by simply setting some of the entries in the partitions equal to zero.

\begin{defn}\label{D:BlockCompressions}
	Let $P=\{1,\ldots,p\}$ and let $R,S\subseteq P$. Let $G\in\BR^{\un{n}\times \un{m}}$ for partitions $\un{n},\un{m}\in\BZ_+^p$. Then $G(R,S)$ denotes the block matrix in $\BR^{\un{n}_R\times \un{m}_S}$ where
	\begin{align*}
		&\un{n}_R=(n_{1,R},\ldots,n_{p,R})\in\BZ_+^p,  \ \mbox{with } n_{j,R}=
		\left\{ \begin{array}{cc} 0 & \mbox{ if $j\not\in R$}\\ n_j & \mbox{ if $j\in R$}\end{array}
		\right.\\
		&\un{m}_S=(m_{1,S},\ldots,m_{q,S})\in\BZ_+^q,  \ \mbox{with } m_{j,S}=
		\left\{ \begin{array}{cc} 0 &\!\! \mbox{ if $j\not\in S$}\\ m_j & \!\! \mbox{ if $j\in S$}\end{array}
		\right.
	\end{align*}
	and where $G(R,S)=[\wtilG_{ij}]_{i,j=1,\ldots,p}$ is defined by
	\begin{align*}
		\wtilG_{ij}=G_{ij} \mbox{ if $i\in R$ and $j\in S$},\quad \mbox{and $\wtilG_{ij}$ vacuous if $i\notin R$ or $j\notin S$.}
	\end{align*}
\end{defn}

If $R$ is a singleton, say $R=\{i\}$, we write $G(i,S)$ and likewise we write $G(R,j)$ if $S=\{j\}$. For one-sided compressions, we follow Matlab notation, and write $G(:,S)$ in case $R=P$, or $G(R,:)$ in case $S=P$.

\begin{ex}
	This example illustrates the ideas in Proposition \ref{P:MultClosed} and Definition \ref{D:BlockCompressions}. Consider the poset $\cP_1$ in Example \ref{ex5posets} and partitions $\underline{n},\underline{m},\un{r}\in\BZ_+^p$ as well as matrices $G\in\cI_{\cP_1}^{\un{n}\times\un{m}}$, $H\in\cI_{\cP_1}^{\un{m}\times\un{r}}$ and $K\in\cI_{\cP_1}^{\un{n}\times\un{n}}$. Looking at the product $GH$, we have
	\begin{align*}
	\begin{bmatrix}
	G_{11} & 0 & 0\\
	G_{21} & G_{22} & 0\\
	G_{31} & 0 & G_{33}
 	\end{bmatrix}
 	\begin{bmatrix}
 	H_{11} & 0 & 0\\
 	H_{21} & H_{22} & 0\\
 	H_{31} & 0 & H_{33}
 	\end{bmatrix} =
 	\begin{bmatrix}
 	G_{11}H_{11} & 0 & 0\\
 	G_{21}H_{11} + G_{22}H_{21} & G_{22}H_{22} & 0\\
 	G_{31}H_{11} +G_{33}H_{31} & 0 & G_{33}H_{33}
 	\end{bmatrix}
	\end{align*}
	which illustrates that $GH\in\cI_{\cP_1}^{\un{n}\times\un{r}}$. Similarly, if the inverse of $K$ exists, then
	\begin{align*}
		\begin{bmatrix}
		K_{11} & 0 & 0\\
		K_{21} & K_{22} & 0\\
		K_{31} & 0 & K_{33}
		\end{bmatrix}^{-1} = \begin{bmatrix}
		K_{11}^{-1} & 0 & 0\\
		- K_{22}^{-1}K_{21}K_{11}^{-1} & K_{22}^{-1} & 0\\
		- K_{33}^{-1}K_{31}K_{11}^{-1} & 0 & K_{33}^{-1}
		\end{bmatrix}	
	\end{align*}
	and again, we see that $K^{-1}\in\cI_{\cP_1}^{\un{n}\times\un{n}}$.
	
	Consider the poset $\cP_5=(P_5,\succeq)$ in Example \ref{ex5posets} and partitions $\un{n},\un{m}\in\BZ_+^p$ as well as a matrix $H\in\mathcal{I}^{\underline{n}\times \underline{m}}_{\cP_5}$ such as in Example \ref{exMatrixZerostructures}. For $\cP_5$, we have ${\downarrow}1=\{1,3,4\}$, ${\uparrow}3=\{1,3\}$. Thus
	\begin{align*}
	H({\downarrow}1,1) = \begin{bmatrix}
	H_{11} \\ H_{31} \\ H_{41}
	\end{bmatrix},\quad
	H({\uparrow}3,:) = \begin{bmatrix}
	H_{11} & 0 & 0 & 0\\
	H_{31} & 0 & H_{33} & 0
	\end{bmatrix} \quad \ands \quad
	H(4,2)=H_{42}.
	\end{align*}
\end{ex}

\begin{thm}\label{thmdownstreammult}
	Given a poset $\cP=(P,\succeq)$ with $P=\{1,2,\ldots,p\}$, partitions $\un{n},\un{m},\un{r}\in\BZ_+^p$ and subsets $Q,S\subseteq P$, for any block matrices $G\in \BR^{\underline{r}\times \underline{n}}$ and $H\in \mathcal{I}^{\underline{n}\times \underline{m}}_{\cP}$ we have
	\[
	(GH)(Q,S) = G(Q,R)H(R,S),\quad \mbox{for any subset $R\subseteq P$ with ${\downarrow}S\subseteq R$.}
	\]
	In particular $(GH)(Q,S) = G(Q,{\downarrow}S)H({\downarrow}S,S)$.
\end{thm}

\begin{proof}[\bf{Proof}]
	The block matrices $(GH)(Q,S)$ and $G(Q,R)H(R,S)$ are equally sized. Therfore, the matrices are equal if the corresponding block entries $(GH)(i,j)$ and $G(i,R)H(R,j)$ are equal for each $i\in Q$ and $j\in S$. For $j,k\in P$, if $k\notin{\downarrow}j$, then $j\not\succeq k$ and thus $H(k,j)=0$, because $H\in\mathcal{I}^{\underline{n}\times \underline{m}}_{\cP}$. Let $i \in Q$ and $j\in S$ be arbitrary. Assume that ${\downarrow}S \subseteq R \subseteq P$. Now ${\downarrow}j\subseteq {\downarrow}S\subseteq R$, so if $k\notin R$, then $k\notin {\downarrow}j$ and hence $H(k,j)=0$. With this in mind, consider the $(i,j)^{\text{th}}$ block entry of the matrix $GH$:
	\begin{align*}
		(GH)(i,j)
		& = \sum_{k=1}^pG(i,k)H(k,j)
		= \sum_{k\notin R}G(i,k)H(k,j) + \sum_{k\in R}G(i,k)H(k,j)\\
		& = 0 + \sum_{k\in R}G(i,k)H(k,j)
		= G(i,R)H(R,j)
	\end{align*}
	which completes the proof.
\end{proof}

Finally, we define the block identity matrix $I_{\un{n}}\in\BR^{\un{n}\times\un{n}}$ with respect to a partition $\un{n}\in\BZ_+^p$ as the block diagonal matrix in $\BR^{\un{n}\times\un{n}}$ with identity matrices as diagonal blocks. Then, for any $S\subseteq P$, the matrix $I_{\un{n}}(:,S)$ can be viewed as the embedding of $\BR^{\un{n}_S}$ into $\BR^{\un{n}}$ and $I_{\un{n}}(S,:)$ as the projection from $\BR^{\un{n}}$ onto $\BR^{\un{n}_S}$.

\begin{cor}\label{C:InvSubs}
	Let $H\in\mathcal{I}^{\underline{n}\times \underline{m}}_{\cP}$. For any $S\subseteq P$, we have $H (\BR^{\un{m}_{{\downarrow}S}}) \subseteq \BR^{\un{n}_{{\downarrow}S}}$. In particular, if $\un{m}=\un{n}$, then $\BR^{\un{m}_{{\downarrow}S}}$ is an $H$-invariant subspace.
\end{cor}

\begin{proof}[\bf Proof]
	Apply Theorem \ref{thmdownstreammult} with $G=I_{\un{n}}$, $Q=P$ and $S={\downarrow}S$. This gives
	\[
	H(:,{\downarrow}S)=(I_\un{n}H) (:,{\downarrow}S) = I_\un{n}(:,{\downarrow}S) H({\downarrow}S,{\downarrow}S).
	\]
	Therefore, we have
	\[
	H (\BR^{\un{m}_{{\downarrow}S}}) =\im B(:,{\downarrow}S) =  \im I_\un{n}(:,{\downarrow}S) H({\downarrow}S,{\downarrow}S) \subseteq \im I_\un{n}(:,{\downarrow}S) =  \BR^{\un{n}_{{\downarrow}S}}.\qedhere
	\]
\end{proof}

\subsection{Poset-causal systems}\label{SS:PosetCausal}

Now that the required preliminaries are dealt with, we can define poset-causal systems.

\begin{defn}\label{D:PosetSys}
	Let $\mathcal{P} = (P,\succeq)$ be a poset with $P=\{1,\ldots,p\}$.  A poset-causal system $\Si_\cP$ (with underlying poset $\cP$) is a linear time invariant system \eqref{eqLTI} with structured system matrices
	\begin{equation}\label{PCincl}
		A\in \cI_\cP^{\un{n}\times\un{n}},\quad
		B\in \cI_\cP^{\un{n}\times\un{m}},\quad
		C\in \cI_\cP^{\un{r}\times\un{n}},\quad
		D\in \cI_\cP^{\un{r}\times\un{m}},
	\end{equation}
	for $\un{n},\un{m},\un{r}\in\BZ_+^p$ and some initial state $x_0\in\cX=\BR^{\un{n}}$.
\end{defn}

Equivalently, the poset-causal system $\Si_\cP$ is given by the interconnected equations \eqref{eqSubsys}, for $i\in P$, determined by the non-zero blocks in the system matrices \eqref{PCincl} and the components of the initial state and with local input, state and output spaces of dimensions
\[
\dim(\cU_i)=m_i,\quad \dim (\cX_i)= n_i,\quad \dim (\cY_i)=r_i.
\]
In short, we will write $\Si_\cP\sim(A,B,C,D,x_0;\cP)$ to indicate the poset-causal system $\Si_\cP$, or even, $\Si_\cP\sim(A,B,C,D)$ when the poset is clear from the context and the initial state is either clear from the context or unspecified.

For a poset-causal systems $\Si_\cP\sim(A,B,C,D)$, since $A(i,j)=0$ if $j\not\succeq i$, for the transpose $A^{\intercal}$ we have $A^{\intercal}(j,i)=0$ if $j\not\succeq i$, that is, if $i\not\succeq_d j$. Consequently, we have
\[
A^\intercal\in \cI_{\cP_d}^{\un{n}\times\un{n}},\quad
B^\intercal\in \cI_{\cP_d}^{\un{m}\times\un{n}},\quad
C^\intercal\in \cI_{\cP_d}^{\un{n}\times\un{r}},\quad
D^\intercal\in \cI_{\cP_d}^{\un{m}\times\un{r}}.
\]
This observation justifies the following definition of the dual system.

\begin{defn}\label{D:DualSys}
	For a poset-causal system $\Sigma_{\mathcal{P}}\sim(A,B,C,D)$, its \emph{dual system} is defined to be the poset-causal system $\Sigma_{\mathcal{P}_d}\sim(A^\intercal,C^\intercal,B^\intercal,D^\intercal)$.
\end{defn}

\begin{rem}\label{R:CLS}
An important subclass of poset-causal systems is that of coordinated linear systems introduced in \cite{RvS08} which is also a subclass of hierarchical systems \cite{FBBMTW80}, see also \cite{vSBKKMPR11,KRvS12,KRvS14a,KRvS14b} for results on system and control theory for coordinated linear systems. As mentioned before, the distinguishing feature of coordinated linear systems is that the Hasse diagram of the poset $\cP$ is an out-tree forrest, or equivalently, the partial order is in-ultra transitive, see definition \ref{D:ultatransitive}. Then, by Lemma \ref{L:in/out-ultra}, the dual system of a coordinated linear system is a poset-causal system with an underlying partial order that is out-ultra transitive, which correspond to Hasse diagrams that are in-tree forrest. In particular, the dual of a coordinated linear system is not a coordinated linear system unless when the partial order happens to be a total order.
\end{rem}

For a finite dimensional space $\cV=\bigoplus_{j\in P} \cV_j$ and $S\subset P$ we define
\begin{equation}\label{eqpartsub}
	\cV_S:=\bigoplus_{j\in S} \cV_j \subseteq \cV.
\end{equation}

In our analysis of poset-causal systems $\Si_\cP$, various derived systems play a role. Firstly, the {\em global system} is just the overall classical state space system \eqref{eqLTI} with state and output given by \eqref{eqintegralform}.

Other derived systems are determined by a fixed choice of $i\in P$. We define the {\em $i$-th local system} where we only consider the impact of the local input $u_i$ on the subsystem $i$:
\begin{equation}\label{eqLocalSys}
	\begin{aligned}
		\dot{x}^{i}(t) &= A_{ii}x^{i}(t)+B_{ii}u_i(t), \qquad {x}^{i}(0)={x}^{i}_0\in \cX_i,\\
		y^{i}(t) &= C_{ii}x^{i}(t)+D_{ii}u_i(t), \qquad t\geq 0.
	\end{aligned}
\end{equation}
Here, the state $x^i$ and output $y^i$ at some final time $t>0$ are given by
\begin{equation}\label{eqLocalSysSol}
	\begin{aligned}
		x^i(t)=x^i(x_0^i,u_i,t) &= e^{A_{ii}(t)} {x}^{i}_0 + \int_{0}^{t}e^{A_{ii}(t-\tau)}B_{ii}u_i(\tau)\dtau,\\
		y^{i}(t)=y^i(x_0^i,u_i,t) &= C_{ii}e^{A_{ii}(t)} {x}^{i}_0 + \int_{0}^{t}C_{ii}e^{A_{ii}(t-\tau)}B_{ii}u_i(\tau)\dtau
		+ D_{ii}u_i(t).
	\end{aligned}
\end{equation}
Next, for the {\em $i$-th downstream system} one considers the impact of the $i$-th input $u_i$ on the local states that are downstream from subsystem $i$:
\begin{equation}\label{eqDownSys}
	\begin{aligned}
		\dot{x}^{{\downarrow}i}(t) &= A({\downarrow}i,{\downarrow}i)x^{{\downarrow}i}(t)+B({\downarrow}i,i)u_i(t), \qquad {x}^{{\downarrow}i}(0)={x}^{{\downarrow}i}_0\in \cX_{{\downarrow} i},\\
		y^{{\downarrow}i}(t) &= C({\downarrow}i,{\downarrow}i)x^{{\downarrow}i}(t)+D({\downarrow}i,i)u_i(t), \qquad t\geq 0.
	\end{aligned}
\end{equation}
Note that the state and output signals, $x^{{\downarrow}i}$ and $y^{{\downarrow}i}$, take values in the spaces $\cX_{{\downarrow} i}$ and $\cY_{{\downarrow} i}$, respectively, defined as in \eqref{eqpartsub}. In this case the state ${x}^{{\downarrow}i}$ and output $y^{{\downarrow}i}$ at some final time $t>0$ are given by
\begin{align}
	x^{{\downarrow}i}(t)=x^{{\downarrow}i}(x_0^{{\downarrow}i},u_i,t)
	&= e^{A({\downarrow}i,{\downarrow}i)t} {x}^{{\downarrow}i}_0 + \int_{0}^{t}e^{A({\downarrow}i,{\downarrow}i)(t-\tau)}B({\downarrow}i,i)u_i(\tau)\dtau, \notag\\
	y^{{\downarrow}i}(t)=y^{{\downarrow}i}(x_0^{{\downarrow}i},u_i,t)
	&= C({\downarrow}i,{\downarrow}i)e^{A({\downarrow}i,{\downarrow}i)t} {x}^{{\downarrow}i}_0 +\label{eqDownSysSol}\\
	& \ \ \ +\int_{0}^{t}C({\downarrow}i,{\downarrow}i)e^{A({\downarrow}i,{\downarrow}i)(t-\tau)}B({\downarrow}i,i)u_i(\tau)\dtau +
	D({\downarrow}i,i) u_i(t).\notag
\end{align}
By Corollary \ref{C:InvSubs} it follows that the system matrices partition as
\[
A=\mat{A({\downarrow}i,{\downarrow}i) & \!* \\ 0 & \!*},\,
B=\mat{B({\downarrow}i,{\downarrow}i) & \!* \\ 0 & \!*},\,
C=\mat{C({\downarrow}i,{\downarrow}i) & \!* \\ 0 & \!*},\,
D=\mat{D({\downarrow}i,{\downarrow}i) & \!* \\ 0 & \!*},
\]
with $*$ indicating unspecified entries.

Lastly, for the {\em $i$-th upstream system}, one considers the impact of the system on the $i$-th output component generated by the subsystems that are upstream of the $i$-th subsystem:
\begin{equation}\label{eqUpSys}
	\begin{aligned}
		\dot{x}^{{\uparrow}i}(t) &= A({\uparrow}i,{\uparrow}i)x^{{\uparrow}i}(t)+B({\uparrow}i,{\uparrow}i)u^{{\uparrow}i}(t), \qquad {x}^{{\uparrow}i}(0)={x}^{{\uparrow}i}_0\in \cX_{{\uparrow} i},\\
		y^{{\uparrow}i}(t) &= C(i,{\uparrow}i)x^{{\uparrow}i}(t)+D(i,{\uparrow}i)u^{{\uparrow}i}(t), \qquad t\geq 0.
	\end{aligned}
\end{equation}
where the input and state signals $u^{{\uparrow}i}$ and $x^{{\uparrow}i}$ take values in $\cU_{{\uparrow}i}$ and $\cX_{{\uparrow}i}$, respectively, while the output signal $y^{{\uparrow}i}$ takes values in $\cY_i$. In this case, applying Corollary \ref{C:InvSubs} to the dual system, it follows that system matrices partition as
\[
A=\mat{A({\uparrow}i,{\uparrow}i) & \!0 \\ * & \!*},\,
B=\mat{B({\uparrow}i,{\uparrow}i) & \!0 \\ * & \!*},\,
C=\mat{C({\uparrow}i,{\uparrow}i) & \!0 \\ * & \!*},\,
D=\mat{D({\uparrow}i,{\uparrow}i) & \!0 \\ * & \!*},
\]
with $*$ indicating unspecified entries. As a consequence we see that the state and output of the $i$-th upstream system are easily obtained from the global system via:
\[
x^{{\uparrow}i}(x_0^{{\uparrow}i},u^{{\uparrow}i},t)=I({\uparrow}i,:)x(x_0,u,t) \ands
y^{{\uparrow}i}(x_0^{{\uparrow}i},u^{{\uparrow}i},t)=y_i(x_0,u,t),
\]
with $y_i$ the $i$-th component of the global output signal $y$ and
where $x_0$ and $u$ can be any initial state and input satisfying
\[
x_0^{{\uparrow}i}=I({\uparrow}i,:) x_0 \ands u^{{\uparrow}i}(t) = I({\uparrow}i,:) u(t).
\]

%Lastly, for the {\em $i$-th upstream system}, one only considers the impact on the subsystem $i$ of local inputs $u_j$ and states $x_j$ of subsystems $j$ which are upstream from $i$. These are given by \eqref{eqSubsys}, and we now rewrite theses equations using block matrix notation:
%\begin{equation}\label{eqUpSysAlt}
%\begin{aligned}
%\dot{x}_i(t) & = A(i,{\uparrow}i)x^{{\uparrow}i}(t) + B(i,{\uparrow}i)u^{{\uparrow}i}(t), \quad x_i(0)=x_{i,0},\\
%y_i(t) & = C(i,{\uparrow}i)x^{{\uparrow}i}(t) + D(i,{\uparrow}i)u^{{\uparrow}i}(t),\quad t\geq 0.
%\end{aligned}
%\end{equation}
%Solving these equations at some final time $t$ gives the final values of the local state $x_i$ and local output $y_i$, because states and inputs of subsystems which are not upstream from $i$ cannot influence $i$.
%
%
%We note that both the $i$-th local subsystem and the $i$-th downstream system can be viewed as classical state space systems in their own right. Hence in order to solve for the state and output signals at some final time $t$, we can simply apply the formulas \eqref{eqintegralform}. However, the $i$-th upstream system  cannot be viewed as a classical state space system (NOT SURE HOW TO PUT THIS) since its state matrix $A(i,{\uparrow}i)$ is generally not square. The following lemma provides solutions for the system \eqref{eqUpSys} at some final time $t>0$ (UNDER SPECIAL CIRCUMSTANCES?)

The relation between the signals of the $i$-th downstream system and the global and local systems is less straightforward.

\begin{lem}\label{L:SysTrajRels}
	Consider a poset-causal system $\Sigma_\cP\sim(A,B,C,D)$, a given input $u=\bigoplus_{i\in P}u_i$ and initial state $x_0 = \bigoplus_{i\in P}x_{i,0}$. Set
	$x_{0}^{{\downarrow}i} :=\bigoplus_{j\in {\downarrow}i} x_{j,0}\in \cX_{{\downarrow}i}$ and
	$\wtil{x}_{0,i} :=I_{\underline{n}}({\downarrow}i,i) x_{i,0}\in \cX_{{\downarrow}i}$ for all $i\in P$. Then
	\begin{equation}\label{eqSolxyDown}
		\begin{aligned}
			x(x_0,u,t) & = \sum_{i\in P} I_{\underline{n}}(:,{\downarrow}i)x^{{\downarrow}i}(\wtil{x}_{0,i},u_i,t) \ands\\
			y(x_0,u,t) & = \sum_{i\in P} I_{\underline{n}}(:,{\downarrow}i)y^{{\downarrow}i}(\wtil{x}_{0,i},u_i,t), \quad t\geq 0.
		\end{aligned}
	\end{equation}
	Furthermore, for all $i\in P$ we have
	\begin{align}
		&x^{{\downarrow}i}_i(x^{{\downarrow}i}_0,u_i,t)=x^i(x_{i,0},u_i,t) \ands  y^{{\downarrow}i}_i(u_i,t)=y^i(x_0^i,u_i,t), \quad t\geq 0,\label{eqxiiyii}
	\end{align}
	where $x^{{\downarrow}i}_j(x^{{\downarrow}i}_0,u_i,t)$ is the component of $x^{{\downarrow}i}(x^{{\downarrow}i}_0,u_i,t)$ taking values in $\cX_j$ and where $y^{{\downarrow}i}_j(x^{{\downarrow}i}_0,u_i,t)$ is the component of $y^{{\downarrow}i}(x^{{\downarrow}i}_0,u_i,t)$ taking values in $\cY_j$. In particular, for all $i\in P$ we have
	\begin{equation}\label{eqxy_ixy^i}
		\begin{aligned}
			&x_i(x_0,u,t)= x^i(x_{i,0},u_i,t) + \sum_{j\in{\dhup}i}x^{{\downarrow}j}_i(\wtil{x}_{j,0},u_j,t),\\
			&y_i(x_0,u,t) = y^i(x_{i,0},u_i,t) + \sum_{j\in{\dhup}i}y^{{\downarrow}j}_i(\wtil{x}_{j,0},u_j,t), \quad t\geq 0.
		\end{aligned}
	\end{equation}
\end{lem}

\begin{proof}[\bf Proof]
	Firstly we note that for any $i\in P$, we have ${\downarrow}i=\{i\}\cup{\dhdown}i$ and that if $j\in {\dhdown}i$, then $j\not\succeq i$. In particular, for $j\in {\dhdown}i$ we have $A_{ij}=0$, $B_{ij}=0$, $C_{ij}=0$ , $D_{ij}=0$ and $e^{A}(i,j)=0$. This implies that
	\begin{equation}\label{eqDownPullThrough}
		\begin{aligned}
			I_{\underline{n}}(i,{\downarrow}i)A({\downarrow}i,{\downarrow}i) = A_{ii}I_{\underline{n}}(i,{\downarrow}i) \ands
			I_{\underline{n}}(i,{\downarrow}i)C({\downarrow}i,{\downarrow}i) = C_{ii}I_{\underline{n}}(i,{\downarrow}i).
		\end{aligned}
	\end{equation}
	The first identity yields $I_{\underline{n}}(i,{\downarrow}i)e^{A({\downarrow}i,{\downarrow}i)t} = e^{A_{ii}t}I_{\underline{n}}(i,{\downarrow}i)$. Furthermore, by Theorem \ref{thmdownstreammult} we have
	\[
	I_{\underline{n}}(i,{\downarrow}i)B({\downarrow}i,i)=B(i,i)=B_{ii} \ands
	I_{\underline{n}}(i,{\downarrow}i)D({\downarrow}i,i)=D(i,i)=D_{ii},
	\]
	while Corollary \ref{C:InvSubs} implies that $ e^{At}I_{\underline{n}}(:,{\downarrow}i) = I_{\underline{n}}(:,{\downarrow}i)e^{A({\downarrow}i,{\downarrow}i)t}$. Next observe that
	\begin{align*}
		x(x_0,u,t)
		& =  e^{A t}x_0 + \int_{0}^{t}e^{A(t-\tau)}Bu(\tau) \dtau\\
		&=e^{A t}\sum_{i\in P} I_\un{n}(:,i)x_{i,0}+\int_{0}^{t}e^{A(t-\tau)}\sum_{i\in P}B(:,i)u_i(\tau) \dtau\\
		&=\sum_{i\in P} e^{A t} I_\un{n}(:,i)x_{i,0} +\int_{0}^{t}e^{A(t-\tau)}B(:,i)u_i(\tau) \dtau \\
		&=\sum_{i\in P} e^{A t}I_\un{n}(:,{\downarrow}i) I_\un{n}({\downarrow}i,i)x_{i,0} +\int_{0}^{t}e^{A(t-\tau)}I_{\underline{n}}(:,{\downarrow}i)B({\downarrow}i,i)u_i(\tau) \dtau \\
		&=\sum_{i\in P} I_\un{n}(:,{\downarrow}i) e^{A({\downarrow}i,{\downarrow}i) t} \wtil{x}_{i,0} +\int_{0}^{t}I_{\underline{n}}(:,{\downarrow}i) e^{A({\downarrow}i,{\downarrow}i)(t-\tau)}B({\downarrow}i,i)u_i(\tau) \dtau \\
		&=\sum_{i\in P} I_\un{n}(:,{\downarrow}i)
		\Bigl(e^{A({\downarrow}i,{\downarrow}i) t} \wtil{x}_{i,0} +\int_{0}^{t} e^{A({\downarrow}i,{\downarrow}i)(t-\tau)}B({\downarrow}i,i)u_i(\tau) \dtau \Bigr) \\
		&=\sum_{i\in P} I_\un{n}(:,{\downarrow}i) x^{{\downarrow}i}(\wtil{x}_{0,i},u_i,t).
	\end{align*}
	Hence the identity for $x(x_0,u,t)$ in \eqref{eqSolxyDown} holds. A similar argument also gives the identity for $y(x_0,u,t)$ in \eqref{eqSolxyDown}.
	
	In order to prove the two identities in \eqref{eqxiiyii}, we consider the $i$-th components of the solutions given in \eqref{eqDownSysSol}:
	\begin{align*}
		x^{{\downarrow}i}_i(x^{{\downarrow}i}_0,u_i,t)
		& = I_{\underline{n}}(i,{\downarrow}i)x^{{\downarrow}i}(x^{{\downarrow}i}_0,u_i,t)\\
		& = I_{\underline{n}}(i,{\downarrow}i)e^{A({\downarrow}i,{\downarrow}i)t} {x}^{{\downarrow}i}_0 + I_{\underline{n}}(i,{\downarrow}i) \int_{0}^{t}e^{A({\downarrow}i,{\downarrow}i)(t-\tau)}B({\downarrow}i,i)u_i(\tau)\dtau \\
		&= e^{A_{ii}t}x_0^i+\int_{0}^{t}e^{A_{ii}(t-\tau)} B_{ii}u_i(\tau)\dtau = x^i(u_i,t).
	\end{align*}
	A similar computation gives the identity for $y^{{\downarrow}i}_i(x_0^{{\downarrow}i},u_i,t)$. The two identities in \eqref{eqxy_ixy^i} follow by combining \eqref{eqSolxyDown} and \eqref{eqxiiyii} noting that in \eqref{eqxiiyii} only the $i$-th component of the initial state $x^{{\downarrow}i}_0$ is relevant, so that $x^{{\downarrow}i}_0$ in the left hand sides of both equations may be replaced by $\wtil{x}_{i,0}$.
\end{proof}

%\newpage
%%%%%%%%%%%%%%%%%%%%%%%%%%%%%%%%%%%%%%%%%%%%%%%%%%%%%%%%%%%%%%%%%%%%%%%%%%%%%%%%%%%%%%%%%%%%%%%%
\section{Downstream reachable states and upstream controllability}\label{S:UpControl}

In this section we investigate various notions of controllability for poset-causal systems that respect the partitioning of the state space and the associated block zero-pattern of the system matrices. Two of these concepts are generalizations of controllability notions that were defined in \cite{KRvS12} for coordinated linear systems with one leader system and two follower subsystems ($\cP_1$ in Example \ref{ex5posets}). In our approach to poset-causal systems, we do not identify leaders and followers, but rather make use of the concept of downstream reachable states.

\subsection{Downstream reachable states}\label{SS:DownReach}

Let $\Si_\cP\sim(A,B,C,D)$ be a poset-causal system. For $i\in P$, the \emph{$i$-downstream reachable set} $\mathcal{R}_i(A,B)$ consists of the states that are reachable in the $i$-th downstream system \eqref{eqDownSys}, i.e., vectors in the subspace $\cX_{{\downarrow} i}$ that are reachable in the $i$-th downstream input-state system
$$
\dot{x}^{{\downarrow} i}(t) = A({\downarrow} i,{\downarrow} i)x^{{\downarrow} i}(t) + B({\downarrow}i,i)u_i(t), \quad x^{{\downarrow} i}(0)=0.
$$
Thus, vectors $\xi$ in $\cR_i(A,B)$ are given by the integral formula
\[
\xi = x^{{\downarrow} i}(0,u_i,t) = \int_{0}^{t}e^{A({\downarrow}i,{\downarrow}i)(t-\tau)}B({\downarrow}i,i)u_i(\tau) \dtau.
\]
Equivalently, the $i$-downstream reachable set is given by
\[
\cR_i(A,B)=\cR(A({\downarrow} i,{\downarrow} i),B({\downarrow} i,i))=\im \mathcal{C}(A({\downarrow} i,{\downarrow} i),B({\downarrow} i,i)).
\]
If $\xi\in\cR_i(A,B)$, then we say that $\xi$ is \emph{$i$-downstream reachable}. We note that $\cR_i(A,B)$ is the smallest $A({\downarrow}i,{\downarrow}i)$-invariant subspace of $\cX_{{\downarrow}i}$ that contains $\im B({\downarrow}i,i)$. In the sequel, when no confusion can arise we will omit $A$ and $B$ in the notation, and simply write $\cR_i$ for $\cR_i(A,B)$, and apply similar relaxations of the notation for derived subspaces defined below.

\begin{lem}\label{L:SumRi}
	For a poset-causal system $\Si_\cP\sim(A,B,C,D)$, we have
	\[
	\cR(A,B)=\sum_{i=1}^pI_{\underline{n}}(:,{\downarrow} i) \mathcal{R}_i.
	\]
\end{lem}

\begin{proof}[\bf{Proof}]
	Fix some final time $t>0$. Then $\xi\in \cR(A,B)$ if and only if there exists some input $u=\bigoplus_{i=1}^pu_i$ such that $\xi = x(0,u,t)$. Now $x^{{\downarrow}i}(0,u_i,t) \in \cR_i$ for all $i\in P$ and by \eqref{eqSolxyDown}, we have
	\[
	\xi = x(0,u,t)  = \sum_{i=1}^p I_{\underline{n}}(:,{\downarrow}i)x^{{\downarrow}i}(0,u_i,t).
	\]
	Hence $\xi \in \cR(A,B)$ if and only if $\xi \in \sum_{i=1}^pI_{\underline{n}}(:,{\downarrow} i) \mathcal{R}_i.$
\end{proof}

\subsection{Upstream controllability}

For each $j\in P$ and $i\in{\downarrow}j$ we define the following subspaces of $\cX_i$:
\[
\overline{\cR}_i^j=\overline{\cR}_i^j(A,B) := \cX_i \cap \mathcal{R}_j(A,B) \ands
\widetilde{\mathcal{R}}_i^j=\widetilde{\mathcal{R}}_i^j(A,B) := P_{\cX_i}\mathcal{R}_j(A,B).
\]
Here $P_{\cX_i}$ is the orthogonal projection onto $\cX_i$. One can view $\overline{\mathcal{R}}_i^j$ as the set of local states $x_i\in \cX_i$ that can be reached from a local input $u_j$ in such a way that the other states downstream from $j$ remain unaffected. The subspace $\widetilde{\mathcal{R}}_i^j$, on the other hand, is the set of local states $x_i \in \cX_i$ that can be reached from a local input $u_j$ while the other states downstream from subsystem $j$ may also be affected. From the definitions of the subspaces $\overline{\mathcal{R}}^j_i$ and $\widetilde{\mathcal{R}}_i^j$, we directly get the following inclusions:
\begin{equation}\label{eqoplusinclusions}
\bigoplus_{i\in{\downarrow}j}\overline{\mathcal{R}}^j_i\subseteq \mathcal{R}_j
\subseteq \bigoplus_{i\in{\downarrow}j}\widetilde{\mathcal{R}}_i^j.
\end{equation}
Next we define subspaces, $\overline{\cR}$, $\cR^\circ$ and $\widetilde{\cR}$, of the state space $\cX$ which respect the structure imposed by the poset $\cP$:
\begin{equation}\label{eqR_joverlinetilde}
\begin{aligned}
\overline{\cR}  := \bigoplus_{j\in P} \overline{\cR}_j, \mbox{ where } \overline{\cR}_j:=	\sum_{i\in{\uparrow}j}\overline{\mathcal{R}}_j^i,\quad
\widetilde{\cR}  :=\bigoplus_{j\in P} \widetilde{\cR}_j, \mbox{ where }
\widetilde{\cR}_j := \sum_{i\in{\uparrow}j} \widetilde{\mathcal{R}}^i_j.
\end{aligned}
\end{equation}
Note that the sums in \eqref{eqR_joverlinetilde} are over upstream sets, while the direct sums in \eqref{eqoplusinclusions} were over downstream sets.

\begin{defn}
We call a poset-causal system $\Sigma_{\mathcal{P}}$ \emph{independently controllable} if $\overline{\cR}=\cX$, and \emph{weakly upstream controllable} if $\widetilde{\cR}=\cX$.
\end{defn}

In the context of coordinated linear systems, what we define as independent controllability above goes by the same name in Definition 3.16 in \cite{KRvS12}. Weak upstream controllability does not appear to have been studied for coordinated linear systems yet, however, the subspaces $\widetilde{\cR}_j$, play a role in Lemma 3.15 of \cite{KRvS12}.

The main reason for studying the spaces $\overline{\cR}$ and $\widetilde{\cR}$ instead of $\cR$, is that they are structured as direct sums of subspaces of the local state spaces $\cX_j$. Hence compressions, restrictions and projections of the system matrices to these subspaces exhibit the same poset-causal structure as the original system matrices. Such subspaces will be called {\em structured}.

Via the observation above, the subspace $\widetilde{\cR}_j$ can be interpreted as the states in $\cX_j$ that can be reached from inputs $u_i$ in the subsystems that are upstream from the $j$-th subsystem while states in the other subsystems (that is, states $x_i$ with $i\neq j$) are allowed to be affected. For $\overline{\cR}_j$ only states in $\cX_j$ are included in case they can be reached from an input $u_i$ of an upstream subsystem (i.e., $i\in{\uparrow}j$) such that no states $x_l$ in local subspaces other than $\cX_j$ are effected.

For theoretical purposes we also introduce the structured subspace $\cR^\circ$ of $\cX$ defined by
\begin{equation}\label{eqRcirc}
\begin{aligned}
\cR^\circ  :=\bigoplus_{j\in P} \cR^\circ_j,\quad \mbox{where}\quad
\cR^\circ_j:=\cX_j \cap \cR.
\end{aligned}
\end{equation}
There does not appear to a clear interpretation of $\cR^\circ$ in terms of the communication structure of the poset-causal system. Its relevance becomes clear from the following theorem, which is the main result of this section.

\begin{thm}\label{T:Contr}
For a poset-causal system $\Sigma_\cP\sim(A,B,C,D)$, we have
\begin{equation}\label{eqContrIncl}
\overline{\cR} \subseteq \cR^\circ \subseteq \cR \subseteq \widetilde{\cR}.
\end{equation}
In particular, if $\Sigma_{\mathcal{P}}$ is independently controllable, then $\Sigma_{\mathcal{P}}$ is controllable and if $\Sigma_{\mathcal{P}}$ is controllable, then $\Sigma_{\mathcal{P}}$ is weakly upstream controllable. Furthermore, if
\[
\cQ=\bigoplus_{j\in P} \cQ_j \quad \text{and} \quad
\cS=\bigoplus_{j\in P} \cS_j  \quad \text{such that} \quad
\cQ\subseteq\cR \subseteq \cS,
\]
where $\cQ_j\subseteq \cX_j$ and $\cS_j\subseteq \cX_j$  for each $j\in P$, then $\cQ\subseteq \cR^\circ$ and $\widetilde{\cR}\subseteq\cS$.
\end{thm}

The last claim of the above theorem can be interpreted as saying that among all structured subspaces of the state space $\cX$, the subspace $\cR^\circ$ is the largest included in the controllable subspace $\cR$ and $\widetilde{\cR}$ is the smallest structured subspace that includes $\cR$. In the context of coordinated linear systems (with underlying structure $\cP_1$ in Example \ref{ex5posets}), the inclusions $\overline{\cR} \subseteq \cR \subseteq \widetilde{\cR}$ were obtained in the proof of \cite[Lemma 3.15]{KRvS12}. That $\overline{\cR}$ is not an maximal structured lower subspace of $\cR$ results from the fact that for subspaces $A$, $B$ and $C$ the subspaces $(A\cap B) + (A\cap C)$ and $A\cap (B+C)$ need not coincide. The observation regarding $\cR^\circ$ leads to the observation that $\ov{\cR}$ is the maximal structured lower subspace of $\cR$ if and only if $\ov{\cR}_j=\cX_j \cap \cR$ for all $j\in P$.

The proof of Theorem \ref{T:Contr} is given later in this subsection, after we have proved the following two intermediate lemmas.

\begin{lem}\label{lemRinclusions}
We have the following inclusions:
\begin{align}
\overline{\cR} \subseteq \cR^\circ\subseteq \cR \subseteq \widetilde{\cR}.\label{eqIncl}
\end{align}
\end{lem}

\begin{proof}[\bf{Proof}]
For the inclusion $\overline{\cR}\subseteq \mathcal{R}^\circ$, note that for all $j\in P$ we have by Lemma \ref{L:SumRi} that
\begin{align*}
\ov{\cR}_j
= \sum_{i\in{\uparrow}j}\overline{\mathcal{R}}_j^i
= \sum_{i\in {\uparrow}j} (\cX_j \cap \cR_i)
& \subseteq \cX_j \cap \Bigl(\sum_{i\in {\uparrow}j} I_{\un{n}}(:,{\downarrow}i) \cR_i\Bigr)\\
& \subseteq \cX_j \cap \Bigl(\sum_{i\in P} I_{\un{n}}(:,{\downarrow}i) \cR_i\Bigr)
= \cX_j \cap \cR
=\cR_j^\circ.
\end{align*}
To prove the second inclusion, $\cR^\circ \subseteq \cR$, we see that
\[
\cR^\circ
= \bigoplus_{j\in P}\cR_j^\circ
= \bigoplus_{j\in P}(\cX_j \cap\cR) \subseteq \cR \cap \bigoplus_{j\in P}\cX_j = \cR\cap\cX = \cR.
\]

For the final inclusion $\cR \subseteq \widetilde{\cR}$, define  $\Phi_1=\{(i,j):j\in P, \ i \in {\uparrow}j\}$ and $\Phi_2=\{(i,j):i\in P, \ j\in{\downarrow}i\}$. Note that $\Phi_1=\Phi_2$, because  $i\in{\uparrow}j$ if and only if $j\in{\downarrow}i$. By Lemma \ref{L:SumRi} and the second inclusion in \eqref{eqoplusinclusions}, we obtain that
	\begin{align*}
	\mathcal{R}
	&=\sum_{i\in P} I_{\underline{n}}(:,{\downarrow}i)\mathcal{R}_i
	\subseteq  \sum_{i\in P}  I_{\underline{n}}(:,{\downarrow}i)  \bigoplus_{j\in{\downarrow}i}\widetilde{\cR}^i_j
	= \sum_{i\in P}  I_{\underline{n}}(:,{\downarrow}i)\sum_{j\in{\downarrow}i}I_\un{n}({\downarrow}i,j)\widetilde{\cR}^i_j\\
	&= \sum_{(i,j)\in \Phi_2}  I_{\underline{n}}(:,{\downarrow}i)I_\un{n}({\downarrow}i,j)\widetilde{\cR}^i_j
	= \sum_{(i,j)\in \Phi_2} I_{\underline{n}}(:,j) \widetilde{\cR}^i_j
	= \sum_{(i,j)\in \Phi_1} I_{\underline{n}}(:,j) \widetilde{\cR}^i_j\\
	&= \sum_{j\in P} I_{\underline{n}}(:,j) \sum_{i\in{\uparrow}j} \widetilde{\cR}^i_j
	=\bigoplus_{j\in P}\widetilde{\cR}_j=\wtil{\cR}.
	\end{align*}
	This completes the proof.
\end{proof}

\begin{lem}\label{L:Rtilde}
For each $j\in P$, we have that
\[
\cR_j^\circ=\cX_j \cap\cR \quad \ands \quad \widetilde{\cR}_j = P_{\cX_j}\cR.
\]
\end{lem}

\begin{proof}[\bf{Proof}]
There is nothing to prove for the first identity. For the identity $\widetilde{\cR}_j = P_{\cX_j}\cR$ we have by definition that $\widetilde{\cR}_j^i = P_{\cX_j}\cR_i$ and we have $\cX_j \perp \cR_i$ if $i\not\in{\uparrow}j$. Thus, by the linearity of the projection  $P_{\cX_j}$, we get that
\begin{align*}
\widetilde{\cR}_j
& = \sum_{i\in{\uparrow}j}\widetilde{\cR}_j^i
= \sum_{i\in{\uparrow}j}P_{\cX_j}\cR_i
= \sum_{i\in{\uparrow}j}P_{\cX_j}\Bigl(I_{\un{n}}(:,{\downarrow}i)\cR_i\Bigr)
= P_{\cX_j}\sum_{i\in{\uparrow}j}I_{\un{n}}(:,{\downarrow}i)\cR_i + \{0\}\\
& = P_{\cX_j}\sum_{i\in{\uparrow}j}I_{\un{n}}(:,{\downarrow}i)\cR_i +
P_{\cX_j}\sum_{i\notin{\uparrow}j}I_{\un{n}}(:,{\downarrow}i)\cR_i
= P_{\cX_j}\sum_{i\in P}I_{\un{n}}(:,{\downarrow}i)\cR_i
= P_{\cX_j}\cR,
\end{align*}
where we have applied Lemma \ref{L:SumRi} in the last step.
\end{proof}

\begin{proof}[\bf Proof of Theorem \ref{T:Contr}]
The inclusion \eqref{eqContrIncl} was proved in Lemma \ref{lemRinclusions}, and the relations between controllability, independent controllability, upstream controllability and weak upstream controllability are a direct consequence.

To see that $\cR^\circ$ and $\widetilde{\cR}$ are optimal, assume that
\[
\cQ=\bigoplus_{j\in P} \cQ_j \ands
\cS=\bigoplus_{j\in P} \cS_j  \quad \text{such that} \quad
\cQ\subseteq\cR \subseteq \cS,
\]
with $\cQ_j\subseteq \cX_j$ and $\cS_j\subseteq \cX_j$ for each $j\in P$. Then, by Lemma \ref{L:Rtilde}, we have
\begin{align*}
\cQ_j & =\cX_j \cap \cQ\subseteq \cX_j \cap \cR =\cR_j^\circ \ands
\widetilde{\cR}_j=P_{\cX_j}\cR\subseteq P_{\cX_j}\cS=\cS_j
\end{align*}
for all $j\in P$. Therefore, by \eqref{eqR_joverlinetilde}, we have $\cQ \subseteq\cR^\circ$ and $\widetilde{\cR}\subseteq \cS$.	
\end{proof}

\subsection{A few examples}

We illustrate the above results with two examples. The first example shows in particular that all inclusions in \eqref{eqIncl} can be strict.

\begin{ex}\label{exLargeEx}
Let $\cP_4$ be the poset given in Example \ref{ex5posets}. Consider the poset-causal system $\Si_{\cP_4}\sim(A,B,0,0)$ with $A\in\cI_{\cP_4}^{\un{n}\times\un{n}}$ and $B\in\cI_{\cP_4}^{\un{n}\times\un{m}}$, with $\un{n}=(2,2,3,4)$ and $\un{m}=(2,1,2,2)$, given by
\begin{align*}
	A \!=\!\! \left[\begin{array}{c c|c c|c c c|c c c c}
	1 & 0 &  &  &  &  &  &  &  & \\
	0 & 0 &  &  &  &  &  &  &  &\\\hline
	1 & 0 & 1 & 0 &  &  &  &  &  &\\
	0 & 0 & 0 & 0 &  &  &  &  &  \\\hline
	 &  &  &  & 0 & 0 & 1 &  &  \\
	 &  &  &  & 0 & 0 & 0 &  &  \\
	 &  &  &  & 0 & 0 & 0 &  &  \\\hline
	0 & 0 & 0 & 0 & 0 & 0 & 0 & 0 & 0 & 0 & 0\\
	1 & 0 & 1 & 0 & 0 & 0 & 0 & 0 & 1 & 0 & 0\\
	0 & 0 & 0 & 0 & 0 & 0 & 1 & 0 & 0 & 0 & 0\\
	0 & 0 & 0 & 0 & 0 &-1 & 0 & 0 & 0 & 0 & 1
	\end{array}\right]
,\quad
	B \!=\!\! \left[
	\begin{array}{c c | c |c c| c c}
	1 & 0 &  &  &  & \\
	0 & 0 &  &  &  & \\\hline
	1 & 0 & 0 &  &  & \\
	0 & 1 & 1 &  &  & \\\hline
	 &  &  & 1 & 0 & \\
	 &  &  & 0 & 1 & \\
	 &  &  & 0 & 0 & \\\hline
	0 & 1 & 0 & 0 & 0 & 0 & 0\\
	0 & 0 & 1 & 0 & 0 & 1 & 0\\
	0 & 0 & 0 & 1 & 0 & 0 & 0\\
	0 & 0 & 0 & 0 & 1 & 0 & 1
	\end{array}\right].
	\end{align*}
(The open white spaces in $A$ and $B$ represent appropriately sized zero matrices.) Then $\cX_1=\text{span}\{e_1,e_2\}$, $\cX_2=\text{span}\{e_3,e_4\}$, $\cX_3=\text{span}\{e_5,e_6,e_7\}$ and $\cX_4=\text{span}\{e_8,e_9,e_{10},e_{11}\}$. So that $\cX = \text{span}\{e_1,\cdots,e_9\}=\BR^9$. Here $e_i$ is the $i$-th standard basis vector in $\BR^9$. Firstly, we note that
\[
{\downarrow}1=\{1,2,4\},\quad {\downarrow}2=\{2,4\},\quad {\downarrow}3=\{3,4\},\quad {\downarrow}4=\{4\}.
\]
Using this, the space of reachable states $\cR=\im\cC(A,B)$ as well as the downstream reachable sets $\cR_i=\im\cC(A({\downarrow}i,{\downarrow}i), B({\downarrow}i,i))$ for $i=1,2,3,4$ can be determined as:
\begin{align*}
	\cR & = \text{span}\{e_1,e_3,e_4,(e_5+e_{10}),e_6,e_8,e_9,e_{11}\} \varsubsetneq \cX\\
	\cR_1 & = \text{span}\{e_1,e_3,(e_4+e_8),e_9\} \varsubsetneq \cX_1\oplus\cX_2\oplus\cX_4\\
	\cR_2 & = \text{span}\{e_4,e_9\} \varsubsetneq \cX_2\oplus\cX_4\\
	\cR_3 & = \text{span}\{(e_5+e_{10}),(e_6+e_{11})\} \varsubsetneq \cX_3\oplus\cX_4\\
	\cR_4 & = \text{span}\{e_9,e_{11}\} \varsubsetneq \cX_4.
\end{align*}
Next we note that the spaces $\overline{\cR}_j^i$ for $i\in P$ and $j\in{\downarrow}i$, are given by:
\begin{align*}
\overline{\cR}_1^1  & = \text{span}\{e_1\}  \varsubsetneq \cX_1, &
\overline{\cR}_2^1  & = \text{span}\{e_3\} \varsubsetneq \cX_2, &
\overline{\cR}_4^1  & = \text{span}\{e_9\} \varsubsetneq \cX_4,\\
\overline{\cR}_2^2  & = \text{span}\{e_4\}  \varsubsetneq \cX_2, &
\overline{\cR}_4^2  & = \text{span}\{e_9\}  \varsubsetneq \cX_4, &
\overline{\cR}_3^3  & = \{0\} \varsubsetneq  \cX_3,\\
\overline{\cR}_4^3  & = \{0\} \varsubsetneq  \cX_4, &
\overline{\cR}_4^4  & = \text{span}\{e_9,e_{11}\} \varsubsetneq \cX_4.
	\end{align*}
From these the spaces $\overline{\cR}_j$ and $\cR_j^\circ$ can be computed using \eqref{eqR_joverlinetilde} and $\wtil{\cR}_j$ can be computed using Lemma \ref{L:Rtilde}:
\begin{align*}
\overline{\cR}_1 & = \text{span}\{e_1\},
& \cR^\circ_1 & = \text{span}\{e_1\},
& \widetilde{\cR}_1 & = \text{span}\{e_1\},\\
\overline{\cR}_2 & = \text{span}\{e_3,e_4\},
& \cR^\circ_2 & = \text{span}\{e_3,e_4\},
& \widetilde{\cR}_2 & = \text{span}\{e_3,e_4\},\\
\overline{\cR}_3 & = \{0\},
& \cR^\circ_3 & = \{e_6\}
& \widetilde{\cR}_3 & = \text{span}\{e_5,e_6\},\\
\overline{\cR}_4 & = \text{span}\{e_9,e_{11}\},
& \cR^\circ_4 & = \text{span}\{e_8,e_{9},e_{11}\},
& \widetilde{\cR}_4 & = \text{span}\{e_8,e_9,e_{10},e_{11}\}.
\end{align*}
Finally, we can calculate $\overline{\cR}$, $\cR^\circ$ and $\widetilde{\cR}$ using \eqref{eqR_joverlinetilde}
\begin{align*}
\overline{\cR}  = \text{span}\{e_1,e_3,e_4,e_9,e_{11}\}, &\quad
\cR^\circ  = \text{span}\{e_1,e_3,e_4,e_6,e_8,e_9,e_{11}\},\\
\wtil{\cR}  = \text{span}\{e_1,e_3&,e_4,e_5,e_6,e_8,e_9,e_{10},e_{11}\}.
\end{align*}
This shows that the following inclusions are all strict:
\[
\{0\} \varsubsetneq \overline{\cR} \varsubsetneq \cR^\circ \varsubsetneq \cR \varsubsetneq \wtil{\cR} \varsubsetneq \cX.
\]
In particular, $\Sigma_{\cP_4}$ is not controllable, neither is it independently  or weakly upstream controllable. Note also that no structured subspaces of $\cX$ can be strictly included in between $\cR^\circ$  and $\cR$ and in between  $\cR$ and $\wtil{\cR}$, confirming the optimality claim of Theorem \ref{T:Contr} for this example.

Finally, note that this example we have
\[
A \ov{\cR}=A\cR^\circ = A \cR=\text{span}\{ e_1,e_3,e_9,e_{11}\}
\text{ and }
A \widetilde{\cR}=\text{span}\{ e_1,e_3,e_9,e_{11}, e_5+e_{10}\}.
\]
Hence $\ov{\cR}$, $\cR^\circ$, $\cR$ and $\wtil{\cR}$ are all invariant subspaces of $A$. For $\cR$ this is true in general, but for the other three this need not always happen, as illustrated in the next example.
\end{ex}

\begin{ex}
Now we consider an example where $\ov{\cR}$, $\cR^\circ$, and $\wtil{\cR}$ are not invariant under $A$. In the context of coordinated linear systems (with poset $\cP_1$ in Example \ref{ex5posets}), for $\ov{\cR}$ and $\wtil{\cR}$ this follows from the controllability decompositions in \cite{KRvS12}. Consider a poset-causal system with poset $\cP_6$ in Example \ref{ex5posets}, where $\un{n}=(1,1,2)$, $\un{m}=(1,1,1)$,
\[
A =\matt{c|c|cc}{1&0&0&0\\ \hline 1&0&0&0\\ \hline 1&0&0&0\\ 0&1&-1&0 }
\ands
B = \matt{c|c|c}{1&0&0\\\hline 0&0&0\\ \hline 0&0&0\\ 0&0&0}.
\]
In this case we have
\[
\overline{\cR}=\cR^\circ=\text{span}\{e_1\},\quad \cR=\text{span}\{e_1, (e_2+e_3) \}\ands \wtil{\cR}=\text{span}\{e_1,e_2,e_3\},
\]
so that
\[
A\ov{\cR}=A \cR^\circ = \text{span}\{e_1+e_2+e_3\}\varsubsetneq \ov{\cR}=\cR^\circ \ands A \wtil{\cR}=\BR^4\varsubsetneq \wtil{\cR}.
\]
\end{ex}

\subsection{Weak local controllability}

We conclude this section with the study of a third controllability notion for poset-causal systems.

\begin{defn}
	We call a poset-causal system $\Sigma_{\mathcal{P}}$ \emph{weakly locally controllable} if
	$$
	\widetilde{\mathcal{R}}_i^i = P_{\cX_i}\mathcal{R}_i=\cX_i \quad \mbox{for each $i\in P$.}
	$$
\end{defn}

Weak local controllability implies that each subsystem of $\Si_\cP$, without external influences, seen as a system in its own right, is a controllable system. For coordinated linear systems it corresponds to Definition 3.10 \cite{KRvS12}.

\begin{lem}\label{L:LocContChar}
A poset-causal system $\Sigma_{\mathcal{P}}$ is weak locally controllable if and only if each local pair $(A_{ii},B_{ii})$ is controllable, that is, if and only if all local subsystems \eqref{eqLocalSys} are controllable.
\end{lem}

\begin{proof}[\bf Proof]
Using \eqref{eqDownPullThrough} and the fact that $I_{\un{n}}(i,{\downarrow}i)B({\downarrow}i,i)=B_{ii}$ it follows for all integers $k\geq0$ that $I_{\un{n}}(i,{\downarrow}i)A({\downarrow}i,{\downarrow}i)^kB({\downarrow}i,i) = A_{ii}^k B_{ii}$. Hence
\begin{align*}
\widetilde{\cR}_i^i
&=I_{\un{n}}(i,{\downarrow}i) \cR_i
=\im I_{\un{n}}(i,{\downarrow}i)\, \cC(A({\downarrow}i,{\downarrow}i),B({\downarrow}i,i))\\
&=\im \mat{B_{ii} & A_{ii} B_{ii} & \cdots & A_{ii}^{n-1} B_{ii}}
=\im \mat{B_{ii} & A_{ii} B_{ii} & \cdots & A_{ii}^{n_i-1} B_{ii}}\\
&=\im \cC(A_{ii},B_{ii})=\cR(A_{ii},B_{ii}).
\end{align*}
It follows that $\widetilde{\cR}_i^i=\cX_i$ if and only if $(A_{ii},B_{ii})$ is a controllable pair.
\end{proof}

We next show that weak local controllability also implies controllability of $\Si_\cP$.

\begin{thm}\label{thmweaklocalcontrollable}
If a poset-causal system $\Sigma_{\mathcal{P}}$ is weakly locally controllable, then it is controllable.
\end{thm}

\begin{proof}[\bf Proof]
Assume that $\Sigma_\cP$ is weakly locally controllable. We show that $\cX=\cR$. Fix a $t>0$. Let $\xi=\bigoplus_{j\in P}\xi_j\in\cX$ with $\xi_j\in \cX_j$. We seek an input $u=\bigoplus_{j\in P}u_j$ with $u_j$ taking values in $\cU_j$ so that $\xi=x(0,u,t)$. For $k=1,2,\ldots,p$, set
\begin{equation}\label{eqLk}
L_k:=\{j\in P: |{\uparrow}j|\leq k \}
\end{equation}
and note that $P = L_p= \bigcup_{k=1}^p L_k$ and $L_k\subseteq L_l$ if $k\leq\ell$. We prove by induction that for $k=1,2,\ldots,p$ there exist an input $u$ so that $\xi_{j}=x_j(0,u,t)$ for all $j\in L_k$.

For $k=1$, if $i\in L_1$, then ${\dhup}i=\emptyset$. Thus by \eqref{eqxy_ixy^i}, for any input $u=\bigoplus_{j\in P}u_j$ we have $x_i(0,u,t)=x^i(0,u_i,t)$ with $x_i$ the state of the $i$-th subsystem \eqref{eqSubsys} and $x^i$ the state of the $i$-th local system \eqref{eqLocalSys}. Hence $x_i$ depends only on $u_i$. Since $\Sigma_\cP$ is weakly locally controllable, for $i\in L_1$ there exist inputs $u_i$ so that $x_i(0,u,t)=x^i(0,u_i,t)=\xi_i$. Set $u_j=0$ for $j\not\in L_1$. Then $u$ is an input with the required property.

Now let $k\geq 1$ and assume we have an input $\widetilde{u}=\bigoplus_{j\in P}\widetilde{u}_j$ so that $x_j(0,\wtil{u},t)=\xi_j$ for all $j\in L_{k}$. If $k=p$ then we are done. Otherwise, set $u_j=\wtil{u}_j$ for $j\not\in R_k:=\{j\in L_{k+1} \colon j\not\in L_k\}$ and $R_p:=\emptyset$. For $i\in L_k$ we have ${\uparrow}i\subseteq L_k$ so that $\xi_i=x_i(0,\wtil{u},t)=x_i(0,u,t)$, irrespectively of the choice of the inputs $u_j$ for $j\in R_k$. It remains to select $u_i$ for $i\in R_{k}$ so that also $\xi_i=x_i(0,u,t)$. Let $i\in R_k$. In that case ${\dhup} i\subseteq L_k$. Hence, for all $j\in {\dhup} i$, the input $u_j$ is fixed. By \eqref{eqxy_ixy^i} in Lemma \ref{L:SysTrajRels}, we have that for any input $u_i$
\begin{align*}
x_i(0,u,t) = x^i(0,u_i,t) + \sum_{j\in{\dhup}i}x^{{\downarrow}j}_i(0,u_j,t),
\end{align*}
independent of the choice of the inputs $u_j$ for $j\in R_k$, $j\neq i$. By assumption, the local system \eqref{eqLocalSys} is controllable. Hence there exists an input $u_i$ so that
\[
x^i(0,u_i,t)=\xi_i - \sum_{j\in{\dhup}i}x^{{\downarrow}j}_i(0,u_j,t),
\]
noting that the right hand side is fixed by our selection of inputs $u_j$ for $j\in L_k$. As observed above, we can select $u_i$ independently of the choice of the inputs $u_j$ for $j\in L_k$ with $j\neq i$. This gives us a way to select the remaining inputs $u_i$ for $i\in R_k$ so that $x_j(0,u,t)=\xi_j$ for all $j\in L_{k+1}$. By proceeding inductively we obtain an input $u$ so that $x_j(0,u,t)=\xi_j$ for all $j\in L_p=P$, which proves our claim.
\end{proof}

For weak local controllability, we only show that it implies controllability, but no inclusion of subspaces. Define $\widehat{\cR}:=\oplus \widetilde{\cR}_i^i$. By Theorem \ref{thmweaklocalcontrollable}, if $\widehat{\cR}=\cX$, then $\cR=\cX$. In view of Theorem \ref{T:Contr}, a natural question is whether $\widehat{\cR}\subseteq\cR$ holds also if  $\cR\neq\cX$. This turns out not to be the case, as shown in the next example.

\begin{ex}
Let $\cP=(P,\preceq)$ with $P=\{1,2\}$ and $1\preceq 2$. Take $\un{n}=(1,1)$ and $\un{m}=(1,1)$ and let $\Si_\cP\sim(A,B,0,0)$ be the leader-follower system with
\begin{minipage}{0.3\textwidth}
\begin{figure}[H]
\centering
\begin{tikzpicture}[line cap=round,line join=round,>=triangle 45,x=1.0cm,y=1.0cm]
\clip(6,2.5) rectangle (8,4.5);
\draw [->,line width=0.6pt] (7,4.) -- (7.,3.);
\begin{scriptsize}

\draw [fill=black] (7,4.) circle (2.5pt);
\draw[color=black] (7,4.3) node {$1$};
\draw [fill=black] (7.,3.) circle (2.5pt);
\draw[color=black] (7,2.7) node {$2$};
\draw[color=black] (7.5,4) node {$\cG_{\cP}^{\downarrow}$};
\end{scriptsize}
\end{tikzpicture}
\hspace*{-2cm}
\end{figure}
\end{minipage}
\begin{minipage}{0.7\textwidth}
\[
A = \left[\begin{array}{c |c}
0 & 0   \\\hline
0 & 0
\end{array}\right]
\ands
B = \left[
\begin{array}{c|c}
1 & 0 \\ \hline
1 & 0
\end{array}\right].
\]
\end{minipage}
Then $\cR = \text{span}\{e_1+e_2\}$, $\cR_1 = \text{span}\{e_1+e_2\}$ and $\cR_2 = \{0\}$. Hence $\wtil{\cR}^1_1 = P_{\cX_1}\cR_1 = \text{span}\{e_1\}$ and $\wtil{\cR}^2_2 = \{0\}$ and so $\widehat{\cR}=\text{span}\{e_1\}$. This shows that $\widehat{\cR}\not\subseteq\cR$.
\end{ex}

It was pointed out in \cite{KRvS12} that, for coordinated linear systems, weak local controllability is necessary and sufficient for pole placement. We now show this is also the case for poset-causal systems. We shall first prove the following lemma. Here and in the sequel, $p_X$ denotes the characteristic polynomial of a square matrix $X$. The following Lemma shows that the characteristic polynomial of a matrix in
$\cI_\cP^{\un{n}\times \un{n}}$ is the product of the characteristic polynomials of its main diagonal blocks.

\begin{lem}\label{L:SpecDec}
If $A=[A_{ij}]\in\cI_\cP^{\un{n}\times \un{n}}$, then
\[
p_A(\la)=\prod_{i\in\cP} p_{A_{ii}}(\lambda),\quad \mbox{so that}\quad
\si(A)=\bigcup_{i\in\cP} \si(A_{ii}).
\]
\end{lem}

\begin{proof}[\bf Proof]
For $k=1,\ldots,p$ define $L_k$ as in \eqref{eqLk} and set $R_k:=\{j\in L_{k+1} \colon j\not\in L_k\}=L_{k+1}/L_k$ and $R_p:=\emptyset$ as in the proof of Theorem \ref{thmweaklocalcontrollable} and recall that for $i\in R_k$ we have ${\dhup}i\subseteq L_{k}$. For $k=1,2,\ldots,p$, set
\[
\widehat{A}_k=A(L_k,L_k) \ands
\widetilde{A}_k=A(R_k,R_k).
\]
Since ${\dhup}i\subseteq L_{k}$ for all $i\in R_k$ and $R_k\cap L_{k}=\emptyset$, we have
\[
\widehat{A}_{k+1}=\mat{\widehat{A}_{k}&0\\ * & \widetilde{A}_k} \ands
\widetilde{A}_k=\bigoplus_{i\in R_k} A_{ii},
\]
with $*$ indicating an unspecified matrix.
It now follows recursively that
\[
p_{\widehat{A}_k}(\la)=\prod_{i\in L_k} p_{A_{ii}}(\lambda),\quad \mbox{so that}\quad
\si(\widehat{A}_k)=\bigcup_{i\in L_k}\si(A_{ii}), \quad
k\in\cP.
\]
This proves out claim, since $L_p=\cP$ and $A=\widehat{A}_p$.
\end{proof}

\begin{prop}\label{P:PolePlace}
	A poset-causal system $\Si_\cP\sim(A,B,C,D)$ is weakly locally controllable if and only if for any monic polynomial $p$ of degree $n=n_1+ \cdots + n_p$ there exists a matrix $F\in\cI_\cP^{\un{m}\times \un{n}}$ so that $\det(\lambda I_\un{n}-(A+BF))=p(\lambda)$.
\end{prop}

\begin{proof}[\bf Proof]
Note that the observations about the structure of $A$ with respect to the subspaces associated with $L_k$ and $R_k$ also apply to $B$ and to any matrix $F\in\cI_\cP^{\un{m}\times \un{n}}$. As a consequence, it follows from Proposition \ref{P:MultClosed} and Lemma \ref{L:SpecDec} that
\begin{equation}\label{PPfact}
p_{(A+BF)}(\la)=\prod_{i\in\cP} p_{(A_{ii}+B_{ii}F_{ii})}(\lambda).
\end{equation}
In case $\Si_\cP$ is weakly locally controllable, by the standard pole placement theorem (cf., \cite[Theorem 2.19]{DP00}), for all monic polynomials $p_i$ for $i\in\cP$, with $\deg(p_i)=n_i$ we can find matrices $F_{ii}$ so that $p_{(A_{ii}+B_{ii}F_{ii})}(\lambda)=p_i(\lambda)$. Now factor  $p(\lambda)=\prod_{i\in\cP} p_i(\lambda)$ with $p_i$ monic and $\deg(p_i)=n_i$, and let $F_{ii}$ be as above. Then the block diagonal matrix $F=\diag_{i\in\cP}(F_{ii})$ is in $\cI_\cP^{\un{m}\times \un{n}}$ and our claim follows by \eqref{PPfact}.

Conversely, assume $\Si_\cP$ is not weakly locally controllable. Then by Lemma \ref{L:LocContChar}, there is a $i\in\cP$ such that the pair $(A_{ii},B_{ii})$ is not controllable. This means that $A_{ii}$ has an uncontrollable eigenvalue, say $\la_0$. But then $\la_0$ is an eigenvalue of $A_{ii}+B_{ii}F_{ii}$ for all matrices $F_{ii}\in\BR^{m_i \times n_i}$. Hence by \eqref{PPfact}, $\la_0$ is an eigenvalue of $A+BF$ for all matrices $F\in\cI_\cP^{\un{m}\times \un{n}}$. Thus, any monic polynomial $p$ with degree $n$ which does not have $\lambda_0$ as a root cannot appear as the characteristic polynomial of $A+BF$.
\end{proof}

Proposition \ref{P:PolePlace} shows that weak local controllability corresponds to pole placement via a structured feedback matrix $F$. In case a poset-causal system is controllable but not weakly locally controllable, it follows that pole placement is still possible, but not always via a structured feedback matrix. We illustrate this in the following example, where we, in fact, show that state feedback stabilizability of the global system (in the classical sense) need not imply that state feedback stabilizability can be achieved by a structured feedback matrix.

\begin{ex}
Consider a poset-causal system $\Si_{\cP_6}\sim(A,B,0,0)$ with $\cP_6$ as in Example \ref{ex5posets}, $\un{n}=(2,2,1)$ and $\un{m}=(2,1,1)$ and with $A\in\cI_{\cP_6}^{\un{n}\times\un{n}}$ and $B\in\cI_{\cP_6}^{\un{n}\times\un{m}}$ given by:
\begin{align*}
	A = \left[\begin{array}{c c|c c|c}
	1 & 0 &  &  & \\
	0 & 0 &  &  & \\\hline
	1 & 0 & 1 & 0 &  \\
	0 & 0 & 0 & 0 & \\\hline
	1 & 0 &-1 & 0 & 1   \\
	\end{array}\right]
	\ands
	B = \left[
	\begin{array}{c c | c |c}
	1 & 0 &  & \\
	0 & 1 &  & \\\hline
	1 & 0 & 1 & \\
	0 & 1 & 0 &\\\hline
	0 & 1 & 1 & 1
	\end{array}\right].
\end{align*}
We have $\cX_1=\text{span}\{e_1,e_2\}$, $\cX_2=\text{span}\{e_3,e_4\}$ and $\cX_3=\text{span}\{e_5\}$. So that $\cX = \text{span}\{e_1,e_2,e_3,e_4,e_5\}=\BR^5$. We note that ${\downarrow}1=\{1,2,3\}$, ${\downarrow}2=\{1,2\}$ and ${\downarrow}3=\{3\}$. Using this, we determine the reachable set $\cR=\im\cC(A,B)$ as well as the downstream reachable sets $\cR_i=\cC(A({\downarrow}i,{\downarrow}i), B({\downarrow}i,i))$ for $i=1,2,3$:
\begin{align*}
\cR = \text{span}\{e_1,e_2&,e_3,e_4,e_5\} = \cX,\\
\cR_1  = \text{span}\{(e_1 + e_3),e_2&,(e_4+e_5)\} \varsubsetneq \cX_1\oplus\cX_2\oplus\cX_3,\\
\cR_2 = \text{span}\{(e_3+e_5)\} \varsubsetneq \cX_2&\oplus\cX_3,\quad
\cR_3  = \text{span}\{e_5\} = \cX_3.
\end{align*}
Next we compute the spaces $\overline{\cR}_i^i=X_i\cap \cR_i$:
\begin{align*}
 \widetilde{\cR}_1^1 = \text{span}\{e_1,e_2\} = \cX_1, \quad
 \widetilde{\cR}_2^2 = \text{span}\{e_3\} \varsubsetneq \cX_2, \quad
 \widetilde{\cR}_3^3 =  \text{span}\{e_5\} = \cX_3.
\end{align*}
Since $\cR=\cX$, the system $\Sigma_{\cP_6}$ is controllable, and hence $A$ can be stabilized via state feedback:\ There exists a matrix $F\in\BR^{4 \times 5}$ so that $A+BF$ has eigenvalues only in the open left hand plane $\BC_-:=\{z\in\BC \colon \re(z)<0\}$. However, $\Sigma_{\cP_6}$ is not weakly locally controllable, because $\widetilde{\cR}_2^2 = \text{span}\{e_3\} \neq \cX_2$. Hence there should not exist a matrix $F\in \cI_\cP^{\un{m}\times\un{n}}$ so that $A-BF$ has eigenvalues only in $\BC_-$. Indeed, for $F=[f_{ij}]\in \cI_\cP^{\un{m}\times\un{n}}$ we have
\[
A+BF = \left[\begin{array}{c c|c c|c}
(1+f_{11}) & f_{12} &  &  & \\
f_{21} & f_{22} &  &  & \\\hline
* & * & (1+f_{33}) & f_{34} &  \\
* & * & 0 & 0 & \\\hline
* & * & * & * & (1+f_{45})   \\
\end{array}\right],
\]
and it follows that $0$ will necessarily be an eigenvalue of $A+BF$.
\end{ex}

%%%%%%%%%%%%%%%%%%%%%%%%%%%%%%%%%%%%%%%%%%%%%%%%%%%%%%%%%%%%%%%%%%%%%%%%%%%%%%%%%%%%%%%%%%%%%%%%
\section{Upstream indistinguishable states and downstream observability}\label{S:DownObserve}

In this section we define notions of distinguishability and observability for poset-causal systems that are dual to the notions of reachability and controllability considered in the previous section. We give the definitions and main results, but without proofs. The results follow directly from duality relations determined in the next section.

For a poset-causal system $\Si_\cP\sim(A,B,C,D)$ and a $i\in P$, in correspondence with \eqref{eqpartsub}, define
$$
\cX_{{\uparrow}i} := \bigoplus_{j\in{\uparrow}i}\cX_j\quad \ands \quad
\cX_{P\backslash{\uparrow}i} := \bigoplus_{j\notin{\uparrow}i}\cX_j.
$$
The $i$-upstream indistinguishable set $\cN_i(C,A)$ consists of the initial states $x^{{\uparrow}i}_0\in X_{{\uparrow}i}$ that cannot be distinguished from 0 using the output of subsystem $i$ only. It follows that $\cN_i(C,A)$ is contained in $\cX_{{\uparrow}i}$ and consists of the states $\xi\in \cX_{{\uparrow}i}$ that are indistinguishable from 0 in the system
\begin{align*}
\dot{x}^{{\uparrow}i}(t) & = A({\uparrow}i,{\uparrow}i)x^{{\uparrow}i}(t), \qquad x^{{\uparrow}i}(0) = \xi \\
y^{{\uparrow}i}(t) & = C(i,{\uparrow}i)x^{{\uparrow}i}(t),
\end{align*}
that is, the $i$-th upstream system \eqref{eqUpSys} with zero inputs.
In this case we say that $\xi$ is \emph{$i$-upstream indistinguishable}. It follows that
\[
\cN_i(C,A) = \cN(C(i,{\uparrow}i)A({\uparrow}i,{\uparrow}i)) = \ker\cO(C(i,{\uparrow}i)A({\uparrow}i,{\uparrow}i)).
\]
Also here we usually write $\cN_i$ rather than $\cN_i(C,A)$ if this does not cause confusion.

The following result is the analogue of lemma \ref{L:SumRi} for upstream indistinguishable sets. In the context of the coordinated linear systems this results corresponds to Lemma 4.2 in \cite{KRvS12}.

\begin{lem}\label{L:capNi}
	For a poset-causal system $\Sigma_{\mathcal{P}}\sim(A,B,C,D)$ we have
	\begin{align*}
	\mathcal{N}	& = \bigcap_{i\in P} \left(\mathcal{N}_i \oplus \cX_{P\backslash {\uparrow}i} \right).
	\end{align*}
\end{lem}

Recall that $\mathcal{N}_i \subseteq \cX_{{\uparrow}i}$ and that $\cX_j\subseteq \cX_{{\uparrow}i}$ if $j\in {\uparrow}i$. For each $j\in{\uparrow}i$, we define
$$
\overline{\mathcal{N}}_i^j=\overline{\mathcal{N}}_i^j(C,A) := \mathcal{N}_i(C,A) \cap \cX_j
\ands \widetilde{\mathcal{N}}_i^j=\widetilde{\mathcal{N}}_i^j(C,A) := P_{\cX_j}\cN_i(C,A).
$$
From these definitions, we immediately get the following inclusions:
\begin{equation}\label{eqNinclusions}
\bigoplus_{j\in{\uparrow}i}\overline{\mathcal{N}}_i^j
\subseteq \mathcal{N}_i
\subseteq \bigoplus_{j\in{\uparrow}i} \widetilde{\mathcal{N}}_i^j,
\end{equation}
In analogy with \eqref{eqR_joverlinetilde} and \eqref{eqRcirc} we define the following structured subspaces of $\cX$:
\begin{equation}\label{eqN_joverlinetilde}
\begin{aligned}
\overline{\cN}:=\bigoplus_{j\in P} \overline{\cN}^j,\quad
\cN^\circ:=\bigoplus_{j\in P} \cN^{\circ j},\quad
\widetilde{\cN}:=\bigoplus_{j\in P} \widetilde{\cN}^j, \quad \mbox{where}\\
\overline{\cN}^j: = \bigcap_{i\in{\downarrow}j}\overline{\mathcal{N}}_i^j,\quad
\cN^{\circ j} := P_{\cX_j} \cN,\quad
\widetilde{\cN}^j: = \bigcap_{i\in {\downarrow}j} \widetilde{\mathcal{N}}^j_i.
\end{aligned}
\end{equation}

\begin{defn}
We call a poset-causal system $\Sigma_{\mathcal{P}}$ \emph{independently observable} if $\widetilde{\cN}=\{0\}$, and  \emph{weakly downstream observable} if $\overline{\cN}=\{0\}$.
\end{defn}

In the context of coordinated linear systems, what we define as independent observability, goes by the same name in Definition 4.17 in \cite{KRvS12}. Downstream observability and weak downstream observability does not appear to have been studied for coordinated linear systems yet, but the subspaces $\wtil{\cN}^j$ play an important role in Lemma 4.16 in \cite{KRvS12}.

The space $\overline{\cN}^j$ may be interpreted as the states in $\cX_j$ that are indistinguishable from each other when observing outputs that are downstream from subsystem $j$ (that is, outputs $y_i$ with $i\in{\downarrow}j)$, while not being indistinguishable from states in other subsystems (that is $x_i$ with $i\neq j$). The space $\wtil{\cN}^j$ consist of states in $\cX_j$ that are indistinguishable from each other when observing outputs that are downstream from subsystem $j$ (that is, outputs $y_i$ with $i\in{\downarrow}j)$, while in this case these states are also allowed to be indistinguishable from other states $x_i$ with $i\neq j$.  There does not seem to be a clear interpretation of the states in the space $\cN^{\circ j}$ in terms of the communication structure of the poset-causal system. Its importance is due to the fact that it turns out to be the optimal structured subspace containing $\cN$, as is shown in the following theorem - the main result of this section.

\begin{thm}\label{T:Obs}
For a poset-causal system $\Sigma_\cP\sim(A,B,C,D)$, we have
\begin{equation}\label{eqObsIncl}
\overline{\cN} \subseteq \cN \subseteq  \cN^\circ \subseteq \widetilde{\cN}
\quad\mbox{so that}\quad
\widetilde{\cN}^\perp \subseteq \cN^{\circ\perp}  \subseteq  \cN^\perp \subseteq \overline{\cN}^\perp
\end{equation}
and
\begin{equation}\label{eqOverlineNj}
\overline{\cN}^j = \cX_j\cap\cN
\quad\mbox{so that}\quad
P_{\cX_j}\cN^\perp = \cX_j\ominus\overline{\cN}^j.
\end{equation}
In particular, if $\Sigma_{\mathcal{P}}$ is independently observable, then $\Sigma_{\mathcal{P}}$ is observable and if $\Sigma_{\mathcal{P}}$ is observable, then $\Sigma_{\mathcal{P}}$ is weakly downstream observable. Furthermore, if
\[
\cQ=\bigoplus_{j\in P} \cQ_j, \quad \text{and} \quad
\cS=\bigoplus_{j\in P} \cS_j  \quad \text{such that} \quad
\cQ\subseteq\cN \subseteq \cS,
\]
where $\cQ_j\subseteq \cX_j$ and $\cS_j\subseteq \cX_j$  for each $j\in P$, then $\cQ\subseteq \overline{\cN}$ and $\cN^\circ\subseteq\cS$.
\end{thm}

The above theorem shows that $\overline{\cN}$ is the largest structured subspace of $\cX$ that is contained in $\cN$ and that $\cN^\circ$ is the smallest structured subspace of $\cX$ which contains $\cN$. We conclude this section with the analogue of weak local controllability.

\begin{defn}
The poset-causal system $\Sigma_\cP$ is called \emph{weakly locally observable} if
$$
\overline{\cN}^i_i = \{0\} \mbox{ for each $i\in P$}.
$$
\end{defn}

The analogues of Lemma \ref{L:LocContChar} and Theorem \ref{thmweaklocalcontrollable} are collected in the following result.

\begin{thm}\label{thmweaklocalobservable}
The poset-causal system $\Sigma_\cP$ is weakly locally observable if and only if each local pair $(C_{ii},A_{ii})$ is observable, that is, if and only if all local systems \eqref{eqLocalSys} are observable. If $\Sigma_\cP$ is weakly locally observable, then it is observable.
\end{thm}

All inclusions in \eqref{eqObsIncl} can be strict and it need not be the case that $\bigoplus_{j\in P} \mathcal{N}^j_j$ contains $\cN$. Examples that prove these claims can be obtained from the examples in the previous section and the duality relations explained in the next section. We present here an extension of Example \ref{exLargeEx} that will be useful in the sequel.

\begin{ex}\label{exObsEx}
	Consider the poset $\cP_4$ given in Example \ref{ex5posets} and the poset-causal system $\Si_{\cP_4}\sim(A,0,C,0)$ with $A\in\cI_{\cP_4}^{\un{n}\times\un{n}}$ and $\un{n}$ as in Example \ref{exLargeEx}, $\un{r}=(1,1,1,1)$  and $C\in\cI_{\cP_4}^{\un{r}\times\un{n}}$ given by
\begin{align*}
C \!& = \!\!\left[\begin{array}{c c|c c|c c c|c c c c}
1 & 0 &  &  &  &  &  &  &  &  &  \\\hline
0 & 1 & 0 & 1 &  &  &  &  &  &  & \\\hline
&  &  &  & 0 & 1 & 0 &  &  &  & \\\hline
1 & 0 & 1 & 0 & 1 & 0 & 0 & 0 & 0 & 1 & 0.
\end{array}\right]
\end{align*}
In this case
\[
\cN  = \text{span}\{(-e_2+e_4),(-e_5+e_{10}),e_8,e_9,e_{11}\}
\]
and the upstream indistinguishable sets are given by
\begin{align*}
\cN_1 & =\text{span}\{e_2\} \varsubsetneq \cX_1,\quad
\cN_2  =\text{span}\{e_1,(-e_2+e_4),e_3\} \varsubsetneq \cX_1\oplus \cX_2,\\
\cN_3 & =\text{span}\{e_5,e_7\} \varsubsetneq  \cX_3,\quad
\cN_4  =\text{span}\{e_2,e_4,(-e_5+e_{10}),e_6,e_8,e_9,e_{11}\} \varsubsetneq \cX.
\end{align*}
One can further compute that
\begin{align*}
\overline{\cN}^1 = \overline{\cN}^2  = \overline{\cN}^3  =  \{0\},& \quad \overline{\cN}^4  = \text{span}\{e_8,e_9,e_{11}\},\\
\wtil{\cN}^1  = \text{span}\{e_2\} = \cN^{\circ 1},
\quad  \wtil{\cN}^2  = \text{span}& \{e_4\} = \cN^{\circ 2},
\quad \wtil{\cN}^3  = \text{span}\{e_5\} = \cN^{\circ 3},\\
\wtil{\cN}^4  = \text{span}\{e_8,e_9,& e_{10},e_{11} \} = \cN^{\circ 4},
\end{align*}
from which it follows that
\begin{align*}
\overline{\cN} = \text{span}\{e_8,e_9,e_{11}\}, \quad
\wtil{\cN} = \text{span}\{e_2,e_4,e_5,e_8,e_9,e_{10},e_{11}\} = \cN^\circ.
\end{align*}
This shows that
\[
\{0\}\varsubsetneq \overline{\cN} \varsubsetneq \cN \varsubsetneq \cN^\circ = \wtil{\cN} \varsubsetneq \cX.
\]
Hence the system is not observable, neither is independently or weakly upstream observable. Furthermore, no structured subspace can be strictly include between $\overline{\cN}$ and $\cN$ or between $\cN$ and $\cN^\circ=\wtil{\cN}$. In particular, unlike in Example \ref{exLargeEx}, here the two subspaces $\ov{\cN}$ and $\wtil{\cN}$ of $\cX$ associated with the poset-causal system are the optimal structured subspaces that are included in $\cN$ and include $\cN$, respectively.
\end{ex}

%%%%%%%%%%%%%%%%%%%%%%%%%%%%%%%%%%%%%%%%%%%%%%%%%%%%%%%%%%%%%%%%%%%%%%%%%%%%%%%%%%%%%%%%%%%%%%%%
\section{Duality}\label{S:Dual}

For classical centralized systems, controllability and observability are related through the duality identities
\[
\cR^d = \cN^\perp \quad \ands \quad \cN^d = \cR^\perp.
\]
Here $\cR^d=\cR(A^\intercal,C^\intercal)$ and $\cN^d=\cN(B^\intercal,A^\intercal)$ are the spaces of reachable and indistinguishable states, respectively, of the dual system. In this section we show that there are similar duality relations for the various notions of controllability and observability introduced in this paper. Such observations were not made in \cite{KRvS12}, since the subclass of poset-causal systems considered there is not closed under duality of the underlying posets.

The following theorem is the main result of this section.

\begin{thm}\label{thmduality}
Let $\Si_\cP\sim(A,B,C,D)$ be a poset-causal system, with dual system $\Sigma_{\mathcal{P}_d}\sim(A_d,B_d,C_d,D_d)$. Define $\overline{\cR}$, $\cR^{\circ}$, $\widetilde{\cR}$ as in \eqref{eqR_joverlinetilde} and \eqref{eqRcirc} and $\overline{\cN}$, $\cN^{\circ}$, $\widetilde{\cN}$ as in \eqref{eqN_joverlinetilde}, and define $(\overline{\cR})^d$, $(\cR^{\circ })^d$, $(\widetilde{\cR})^d$, $(\overline{\cN})^d$, $(\cN^{\circ})^d$, $(\widetilde{\cN})^d$ analogously for $\Sigma_{\mathcal{P}_d}$. Then
\[
(\overline{\cR})^d=\widetilde{\cN}^\perp\!\!\!\!\!,\ \ \
(\cR^{\circ})^d=\cN^{\circ \perp}\!\!\!\!\!\!\!,\ \ \
(\widetilde{\cR})^d=\overline{\cN}^\perp\!\!\!\!\!,\ \ \
(\overline{\cN})^d=\widetilde{\cR}^\perp\!\!\!\!,\ \ \
(\cN^{\circ})^d=\cR^{\circ \perp}\!\!\!\!\!\!,\ \ \
(\widetilde{\cN})^d=\overline{\cR}^\perp \!\!.
\]
In particular, the following equivalences hold:
\begin{enumerate}
\item[(i)] $\Sigma_{\mathcal{P}}$ is upstream controllable if and only if $\Sigma_{\mathcal{P}_d}$ is downstream observable.
		
%\item[(ii)] $\Sigma_{\mathcal{P}}$ is weakly upstream controllable if and only if $\Sigma_{\mathcal{P}_d}$ is weakly downstream observable.
		
\item[(ii)] $\Sigma_{\mathcal{P}}$ is weakly locally controllable if and only if $\Sigma_{\mathcal{P}_d}$ is weakly locally observable.
		
	\end{enumerate}
\end{thm}

The identities in Theorem \ref{thmduality} will be proved via several intermediate steps.

An essential role in our definitions of controllability and observability is played  by the downstream reachable and upstream unobservable sets $\cR_i$ and $\cN_i$ respectively.  The next lemma explains the connection of the two sets under duality. Here, we denote the  downstream reachable and upstream unobservable sets of the dual system $\Si_{\cP_d}$ by $(\cR_i)^d$ and $(\cN_i)^d$ respectively.

\begin{lem}\label{lemRiperp=Ndi}
Let $\Si_\cP\sim(A,B,C,D)$ be a poset-causal system, with dual system $\Sigma_{\mathcal{P}_d}\sim(A_d,B_d,C_d,D_d)$. Then
\[
\cX_{{\uparrow}i}\ominus (\cR_i)^d = \cN_i \quad \ands \quad
\cX_{{\downarrow}i}\ominus (\mathcal{N}_i)^d = \cR_i\quad \mbox{for each}\quad i\in P.
\]
\end{lem}

\begin{proof}[\bf Proof]
%By definition $A_d = A^\intercal$, $B_d = C^\intercal$, $C_d = B^\intercal$ and $D_d = D^\intercal$. Thus,
%\[
%\cR^d
%= \im\cC(A_d,B_d)
%= \im\cC(A^\intercal,C^\intercal)
%= \im \cO\left(C,A\right)^\intercal
%= \ker \cO\left(C,A\right)^\perp
%= \cN^\perp.
%\]
%The identity $\cN^d = \cR^\perp$ follows similarly.
%
Fix a $i\in P$. Note that $\cN_i\subseteq \cX_{{\uparrow}i}$ and $(\cR_i)^d\subseteq \cX_{{\downarrow_d}i}=\cX_{{\uparrow}i}$. Then
\begin{align*}
(\cR_i)^d & =\cR(A_d({\downarrow}_d i,{\downarrow}_d i),B_d({\downarrow}_d i,i))
=\cR(A^{\intercal}({\uparrow} i,{\uparrow} i),C^{\intercal}({\uparrow} i,i))\\
&=\cR(A({\uparrow} i,{\uparrow} i)^{\intercal},C(i,{\uparrow} i)^{\intercal}).
\end{align*}
By the standard duality identity, we have
\begin{align*}
\cX_{{\uparrow}i}\ominus (\cR_i)^d  & = \cX_{{\uparrow}i}\ominus \cR(A({\uparrow} i,{\uparrow} i)^{\intercal},C(i,{\uparrow} i)^{\intercal})
=\cN(C(i,{\uparrow} i), A({\uparrow} i,{\uparrow} i))=\cN_i.
\end{align*}
The identity $\cX_{{\downarrow}i}\ominus (\mathcal{N}_i)^d = \cR_i$ follows similarly.
\end{proof}

The relations between the subspaces $\overline{\cR}_i^j$, $\wtil{\cR}_i^j$, $\overline{\cN}_j^i$, $\wtil{\cN}_j^i$ and the related subspaces for the dual system, denoted $(\overline{\cR}_i^j)^d$, $(\wtil{\cR}_i^j)^d$, $(\overline{\cN}_j^i)^d$, $(\wtil{\cN}_j^i)^d$, respectively, is less straightforward. They are listed in Lemma \ref{L:RijNjiDual}, the proof of which relies on some general identities in finite dimensional inner product spaces. Let $\cY_1,\cY_2, \ldots,\cY_n$ be subspaces of a finite dimensional inner product space $\cY$. Then
\begin{equation}\label{eqSubId1}
\Bigl(\bigcap_{i=1}^n\cY_i\Bigr)^\perp = \sum_{i=1}^n\cY_i^\perp.
\end{equation}
This follows from extending the well known and easily proved identity $(\cY_1\cap \cY_2)^\perp = \cY_1^\perp + \cY_2^\perp $. Since we work in finite dimensional spaces, we have $(\cY_i^\perp)^\perp=\cY_i$, and thus \eqref{eqSubId1} also gives us
\begin{equation}\label{eqSubId2}
\bigcap_{i=1}^n\cY_i^\perp = \Bigl(\sum_{i=1}^n\cY_i\Bigr)^\perp.
\end{equation}
Moreover, we also have
\begin{equation}\label{eqSubsId2}
P_{\cY_2}(\cY_1^\perp)=\cY_2 \ominus (\cY_1\cap \cY_2),
\end{equation}
where $P_{\cY_2}$ denotes te orthogonal projection in $\cY$ onto $\cY_2$. This identity is less straightforward and we include a proof.  Using the first identity we find that
\[
\cY_1^\perp + \cY_2^\perp = (\cY_1\cap \cY_2)^\perp = \cY_2\ominus (\cY_1\cap \cY_2) \oplus \cY_2^\perp.
\]
Projecting onto $\cY_2$ on both sides yields
\[
P_{\cY_2}(\cY_1^\perp)=P_{\cY_2} (\cY_2\ominus (\cY_1\cap \cY_2) \oplus \cY_2^\perp)= \cY_2\ominus (\cY_1\cap \cY_2),
\]
as claimed.

\begin{lem}\label{L:RijNjiDual}
	Let $\Si_\cP\sim(A,B,C,D)$ be a poset-causal system, with dual system $\Sigma_{\mathcal{P}_d}\sim(A_d,B_d,C_d,D_d)$. Then for all $i,j\in P$ we have
	\begin{align*}
		\bigl(\wtil{\cN}_j^i\bigr)^d=\cX_i \ominus \overline{\cR}_i^j, \ \ \bigl(\wtil{\cR}_i^j\bigr)^d=\cX_i \ominus \overline{\cN}_j^i,\ \
		\bigl(\overline{\cN}_j^i\bigr)^d=\cX_i\ominus \wtil{\cR}_i^j, \ \
		\bigl(\overline{\cR}_i^j\bigr)^d = \cX_i\ominus \wtil{\cN}_j^i.
	\end{align*}
\end{lem}

\begin{proof}[\bf Proof]
Using Lemma \ref{lemRiperp=Ndi} along with \eqref{eqSubsId2} with $\cY=\cX_{{\downarrow}j}$, $\cY_1=\cR_j$ and $\cY_2=\cX_i$ yields
\begin{align*}
\bigl(\wtil{\cN}_j^i\bigr)^d & = P_{\cX_i}(\cN_j)^d=P_{\cX_i}(\cX_{{\downarrow}j}\ominus\cR_j)=\cX_i\ominus (\cX_i\cap \cR_j)=\cX_i \ominus \overline{\cR}_i^j.
\end{align*}
A similar argument proves the identity $\bigl(\wtil{\cR}_i^j\bigr)^d=\cX_i \ominus \overline{\cN}_j^i$. Using the identity \eqref{eqSubId2}, with $n=2$, $\cY=\cX_{{\downarrow}j}$,  $\cY_1=(\cN_j)^d$ and $\cY_2=\cX_i$, gives
\begin{align*}
\bigl(\overline{\cN}_j^i\bigr)^d
&= \cX_i \cap (\cN_j)^d =(\cX_{{\downarrow}j}\ominus \cX_{{\downarrow}j\backslash i}) \cap (\cX_{{\downarrow}j}\ominus \cR_j)=\cX_{{\downarrow}j}\ominus(\cX_{{\downarrow}j\backslash i} + \cR_j)\\
&= \cX_{{\downarrow}j}\ominus (P_{\cX_i} \cR_j \oplus \cX_{{\downarrow}j\backslash i})= \cX_i \ominus (P_{\cX_i} \cR_j)=\cX_i \ominus\widetilde{\cR}_i^j.
\end{align*}
The identity $\bigl(\overline{\cR}_i^j\bigr)^d = \cX_i\ominus \wtil{\cN}_j^i$ can be derived analogously.
\end{proof}

\begin{cor}
The poset-causal system $\Sigma_\cP$ is weakly locally controllable (weakly locally observable) if and only if the dual system $\Sigma_{\cP_d}$ is weakly locally observable (weakly locally controllable).
\end{cor}

\begin{proof}[\bf Proof]
By Lemma \ref{L:RijNjiDual} it follows that
\[
\bigoplus_{i\in P} \bigl(\widetilde{\cR}_i^i\bigr)^d = \Bigl(\bigoplus_{i\in P} \overline{\cN}_i^i \Bigr)^\perp \ands
\bigoplus_{i\in P} \bigl(\overline{\cN}_i^i\bigr)^d = \Bigl(\bigoplus_{i\in P} \widetilde{\cR}_i^i \Bigr)^\perp,
\]
from which we immediately obtain the result.
\end{proof}

We now prove duality results for $(\overline{\cR}_j)^d$, $(\widetilde{\cR}_j)^d$, $(\cR_j^{\circ})^d$, $(\overline{\cN}^j)^d$, $(\cN^{j \circ })^d$ and $(\widetilde{\cN}^j)^d$.

\begin{lem}\label{L:FurtherDual}
	Let $\Si_\cP\sim(A,B,C,D)$ be a poset-causal system, with dual system $\Sigma_{\mathcal{P}_d}\sim(A_d,B_d,C_d,D_d)$. Then for all $j\in P$ we have
	\begin{align*}
	\bigl( \overline{\cR}_j\bigr)^d = \cX_j \ominus \widetilde{\cN}^j, & &
	\bigl(\cR_j^\circ\bigr)^d = \cX_j \ominus \cN^{\circ j}, & &
	\bigl(\widetilde{\cR}_j\bigr)^d = \cX_j \ominus \overline{\cN}^j,\\
	\bigl(\overline{\cN}^j\bigr)^d = \cX_j \ominus \widetilde{\cR}_j, & &
	\bigl(\cN^{\circ j}\bigr)^d = \cX_j \ominus \cR^\circ_j, &  &
	\bigl(\widetilde{\cN}^j\bigr)^d = \cX_j \ominus \overline{\cR}_j.
	\end{align*}
\end{lem}

\begin{proof}[\bf{Proof}]
By \eqref{eqR_joverlinetilde}, Lemma \ref{L:RijNjiDual} and \eqref{eqSubId1}, we have
\begin{align*}
\bigl(\wtil{\cR}_j\bigr)^d
= \sum_{i\in {\uparrow}_d j} \bigl(\wtil{\cR}_j^i\bigr)^d
= \sum_{i\in {\downarrow} j} \bigl(\cX_j \ominus \overline{\cN}_i^j\bigr)
= \cX_j \ominus \bigcap_{i\in{\downarrow}j}\overline{\cN}_i^j
= \cX_j \ominus \overline{\cN}^j.
\end{align*}
The identities for $\bigl(\overline{\cR}_j\bigr)^d$, $\bigl(\widetilde{\cR}_j\bigr)^d$ and $\bigl(\widetilde{\cN}^j\bigr)^d$ follow in a similar manner. The identity for $\bigl(\cR_j^{\circ }\bigr)^d$ follows from Lemma \ref{lemRiperp=Ndi} and the identity \eqref{eqSubId2}:
\begin{align*}
\bigl(\cR_j^{\circ }\bigr)^d
&=\cX_j \cap \cR^d = \cX_{P\backslash j}^\perp \cap \cN^{\perp} =(\cX_{P\backslash j} + \cN)^\perp =( \cX_{P\backslash j} \oplus P_{\cX_j}\cN)^\perp\\
&= \cX_j \ominus \cN^{\circ j}.
\end{align*}
The identity for $\bigl(\cN^{j \circ }\bigr)^d$ follows similarly.
\end{proof}

Theorem \ref{thmduality} now follows directly from the identities in Lemma \ref{L:FurtherDual}.
Note that Theorems \ref{T:Obs} and \ref{thmweaklocalobservable} in Section \ref{S:DownObserve} also follow directly from the result obtained in this section.

%\newpage
%%%%%%%%%%%%%%%%%%%%%%%%%%%%%%%%%%%%%%%%%%%%%%%%%%%%%%%%%%%%%%%%%%%%%%%%%%%%%%%%%%%%%%%%%%%%%%%%
\section{Minimality and Kalman reduction for poset-causal systems}\label{S:Kalman}

The concepts and theory of minimality  for poset-causal systems are problematic due to the additional structure in the form of the prescribed zero-block structure and the state space decomposition.

We first review the classical setting, before considering the case of poset-causal systems.

\subsection{The unstructured case}

If a classical state space system $\Si\sim(A,B,C,D)$ as in \eqref{eqLTI} is not minimal, one way of obtaining a minimal system that has the same input-output map goes through the Kalman decomposition, cf., \cite{DP00}. Define the subspaces
\begin{align*}
\cX_{co} := \cR\ominus(\cR\cap\cN)=P_{\cR}(\cN^\perp),\quad
\cX_{c\overline{o}} := \cR\cap\cN,\\
\cX_{\overline{c}o} := (\cR+\cN)^\perp, \quad
\cX_{\overline{co}} := \cN\ominus(\cR\cap\cN)=P_{\cN}(\cR^\perp).
\end{align*}
%\begin{align*}
%\cX_{co} = \cR\ominus(\cR\cap\cN)^\perp\!\!\!\!, \ \
%\cX_{c\overline{o}} = \cR\cap\cN\!, \
%\cX_{\overline{c}o} = (\cR+\cN)^\perp\!\!\!\!, \!\ands\!\!
%\cX_{\overline{co}} = \cN\cap(\cR\cap\cN)^\perp\!\!\!\!.
%\end{align*}
The alternative formulas given for $\cX_{co}$ and $\cX_{\overline{co}}$ follow from \eqref{eqSubsId2}. With the above subspaces of $\cX$ we obtain the following orthogonal sum decompositions:
\begin{align*}
\cX_{co}\oplus\cX_{c\overline{o}} =\cR, \quad
\cX_{c\overline{o}}\oplus\cX_{\overline{co}}& =\cN, \quad
\cX_{co}\oplus\cX_{c\overline{o}} \oplus \cX_{\overline{co}} = \cR+\cN \\
\cX = \cX_{co} \oplus \cX_{c\overline{o}} & \oplus \cX_{\overline{c}o} \oplus \cX_{\overline{co}}
\end{align*}
Since $\cR$ and $\cN$ are the smallest invariant subspaces of $A$ that contain $\im B$ and $(\im C)^\perp$, respectively, with respect to this decomposition of the state space, the system matrices $A$, $B$ and $C$ decompose in what is known as the {\em Kalman decomposition} of $\Si$:
\begin{align}\label{eqclassicKalman}
A &  = \begin{bmatrix}
A_{11} & 0 & A_{13} & 0\\
A_{21} & A_{22} & A_{23} & A_{24}\\
0 & 0 & A_{33} & 0\\
0 & 0 & A_{43} & A_{44}
\end{bmatrix}, \quad
B = \begin{bmatrix}
B_1\\
B_2\\
0\\
0\\
\end{bmatrix},\quad
C  = \begin{bmatrix}
C_1 & 0 & C_3 & 0
\end{bmatrix}.
\end{align}
Furthermore, the {\em Kalman reduction} of $\Si$, i.e., the state space system $\Si_\text{min}\sim(A_{11},B_1,C_1,D)$ is minimal and provides the same input-output map as $\Si$ (when in both cases the initial state is 0), because the moments of the two systems coincide:
\[
CA^{k}B=C_1 A_{11}^k B_1,\quad k=0,1,\ldots.
\]

In this paper we consider systems with additional structure, as a result of which we have to consider subspaces that are larger or smaller than $\cR$ and $\cN$ to maintain the structure. In the setting of the Kalman reduction, one can compress the system to a subspace of $\cX$ which contains $\cX_{co}$ in such a way that the moments are maintained. The next lemma provides a suggestion for such a subspace.

\begin{lem}\label{L:EquivReal}
Consider a state space system $\Si\sim(A,B,C,D)$ with reachable space $\cR$ and unobservable space $\cN$. Suppose $\cR'\subseteq\cR\subseteq\cR''$ and $\cN'\subseteq\cN$ are subspaces, and define $\cX'_1:=\cR''\ominus (\cR'\cap\cN')$. Then $\cX_{co} \subseteq \cX'_1$ and if $A'$, $B'$ and $C'$ are the compressions of $A,B$ and $C$ to $\cX'_1$, then
\[
CA^kB = C'A'^kB'\quad k=0,1,\ldots
\]
\end{lem}

\begin{proof}[\bf{Proof}]
	By definition, $\cX_{co} = \cR\ominus(\cR\cap\cN)$. Since, $\cR\subseteq\cR''$ and $(\cR'\cap\cN')\subseteq(\cR\cap\cN)$, it follows that $\cX_{co}\subseteq\cX'$.
	
	For the second part, by analogy of the Kalman decomposition, define the subspaces
		\begin{align*}
		\cX_1' := \cR''\ominus(\cR'\cap\cN'),\
		\cX_2' := \cR'\cap\cN'\!\!,\ \
		\cX_3' := (\cR''+\cN')^\perp\!\!\!\!,\ \
		\cX_4' := \cN'\ominus(\cR'\cap\cN').
 		\end{align*}
		Then we have the following orthogonal sum decompositions:
		\begin{align*}
		\cX_1'\oplus\cX_2' & = \cR''\!\!\!, &
		\cX_2\oplus\cX_4  & = \cN'\!\!,\\
		\cX_1'\oplus\cX_2'\oplus\cX_4'& = \cR''+\cN'\!\!, &
		\cX_1'\oplus\cX_2'\oplus \cX_3'\oplus\cX_4' & = \cX.
		\end{align*}
		Since $\cR'\subseteq\cR$ and $\cN'\subseteq\cN$, we have $\cX_1'\subseteq\cX_{c\overline{o}}$. Also $\cR\subseteq\cR''$. Thus
		\[
		\cR''=\cR\oplus\cZ_1 \ands
		\cX_{c\overline{o}} = \cX_2'\oplus\cZ_2
		\]
		where $\cZ_1=\cR''\ominus \cR$ and $\cZ_2=\cX_{c\overline{o}}\ominus \cX_2'$. This leads to the following decomposition of $\cX$:
		\begin{align*}
		\cX
		& = \cR''\oplus\cR''^\perp
		 = (\cR \oplus \cZ_1) \oplus\cR''^\perp
		 = (\cX_{co}\oplus\cX_{c\overline{o}}) \oplus \cZ_1 \oplus \cR''^\perp \\
		& = \cX_{co}\oplus \cX_2'\oplus\cZ_2 \oplus \cZ_1 \oplus \cR''^\perp.
		\end{align*}
		With respect to this decomposition of $\cX$, the matrices $A$, $B$ and $C$ decompose as:
		\begin{align*}
			A & = \left[\begin{array}{c c c | c c}
			A_{11} & 0 & 0 & A_{14} & A_{15} \\
			A_{21} & A_{22} & A_{23} & A_{24} & A_{25}\\
			A_{31} & A_{32} & A_{33} & A_{34} & A_{35}\\ \hline
			0 & 0 & 0 & A_{44} & A_{45}\\
			0 & 0 & 0 & A_{54} & A_{55}
			\end{array}\right], \quad
			B = \left[\begin{array}{c}
			B_1\\B_2\\B_3\\ \hline0\\0
			\end{array}\right],\\
		 &\qquad\qquad	C  = \left[\begin{array}{c c c | c c}
			C_1 & 0 & 0 & C_4 & C_5
			\end{array}\right].
		\end{align*}
Note that there is no poset-causal structure in this system, hence the indices do not refer to subsystems here. The left bottom zero block in $A$ and the zeroes in $B$ are due to the fact that $\cR = \cX_{co}\oplus\cX_2'\oplus\cZ_2$ is an $A$-invariant subspace of $\cX$ that contains $\im B$. The two zeroes in the left upper block of $A$ and the zeroes in $C$ are due to the fact that $\cX_{c\overline{o}} = \cX_2'\oplus \cZ_2 \subseteq \cN$, which is an $A$-invariant subspace of $\cX$ that contains $(\im C)^\perp$.
		
		Now we compress $A$, $B$ and $C$ to the subspace $\cX'_1$, which is given by
		\begin{align*}
		\cX_1'
		& = \cR'' \ominus (\cR'\cap\cN')
		 = (\cR \oplus \cZ_1) \ominus \cX_2'
= (\cX_{co}\oplus\cX_2'\oplus \cZ_2 \oplus \cZ_1) \ominus \cX_2'\\
		 &= \cX_{co} \oplus \cZ_2 \oplus \cZ_1,
		\end{align*}
		which yields the matrices $A'$, $B'$ and $C'$:
		\begin{align*}
		A'& = \begin{bmatrix}
		A_{11} & 0 & A_{14}\\
		A_{31} & A_{33} & A_{34}\\
		0 & 0 & A_{44}
		\end{bmatrix}, \quad
		B' = \begin{bmatrix}
		B_1\\B_3\\0
		\end{bmatrix},\quad
		C' = \begin{bmatrix}
		C_1 & 0 & C_4
		\end{bmatrix}.
		\end{align*}
It follows that $A'^k$ has the form
\[
A'^k = \begin{bmatrix}
		A_{11}^k & 0 & *\\
		* & A_{33}^k & *\\
		0 & 0 & A_{44}^k
		\end{bmatrix},
\]
with $*$ indicating unspecified entries. Therefore, we now see that
		\[
		CA^kB = C_1A_{11}^kB_1 = C'A'^kB'\!\!, \quad k = 0,1,\ldots \qedhere
		\]
\end{proof}

\begin{rem}
In Lemma \ref{L:EquivReal} we take $\cR''\ominus (\cR'\cap\cN')$, for subspaces $\cR'\subseteq\cR\subseteq\cR''$ and $\cN'\subseteq\cN$, as an upper bound for $\cX_{co}=\cR\ominus (\cR\cap\cN)$. The alternative formula $\cX_{co}=P_\cR(\cN^\perp)$ suggests we could also consider the subspace $P_{\cR''}(\cN'^\perp)=\cR''\ominus (\cR''\cap\cN')$. However, for this choice, the inclusion $\cX_{co}\subseteq P_{\cR''}(\cN'^\perp)$ need not hold. For instance, one can construct a system with $\cX=\BR^3$, $\cR=\text{span}\{e_1+e_2+e_3\}$ and $\cN=\text{span}\{e_1,e_2-e_3\}$, in which case $\cN^\perp = \{e_2+e_3\}$ and $\cX_{co}=P_{\cR}(\cN^\perp) = \cR$. Taking $\cR'' =\text{span}\{e_1+e_2,e_3\}$ and $\cN'=\cN$, we find that  $P_{\cR''}(\cN'^\perp) = \text{span}\{\frac{1}{2}\sqrt{2}(e_1+e_2)+ e_3\}$ which does not contain $\cX_{co}$.
\end{rem}

%IS THE FOLLOWING TRUE? PROBABLY NOT. Say $v_1,\ldots,v_m$ forms an orthonormal basis for $\cX_{co}$ and $w\perp \cX_{co}$, then I think it might be possible that something like $\wtil{\cX}=\text{span}\{v_1+w,v_2,\ldots,v_m\}$ could still work. What might work is that the strict inclusion $\what{\cX}\subsetneq \cX_{co}$ implies that the moments do not coincide. Maybe $P_{\cX_{co}}\wtil{\cX}\neq \cX_{co}$ is sufficient (and that might also be necessary, but I don't have a proof for either).
%
%\begin{lem}\label{lemcontainsXco}
%Consider a system $\Si\sim(A,B,C,D)$ with state space $\cX$. Let $\what{\Si}\sim(\what{A},\what{B},\what{C},D)$ be the system obtained by compressing $A$, $B$ and $C$ to the subspace $\what{\cX}\subset \cX$, resulting in $\what{A}$, $\what{B}$ and $\what{C}$, respectively. If
%\[
%CA^kB = \what{C}\what{A}^k\what{B},\quad k=0,1,\ldots
%\]
%then $\cX_{co}\subseteq\what{\cX}$.
%\end{lem}

\subsection{Kalman reduction for poset-causal systems}

If the system $\Si$ comes with the additional structure of a poset-causal system, i.e., $\Si=\Si_\cP$ for some poset $\cP$, then, in general, the poset structure is lost when $\Si_\cP$ is compressed to the Kalman reduction $\Si_\text{min}$, and one may have to compress to a larger structured subspace of the state space in order to preserve the poset structure.

\begin{defn}
Consider a poset-causal system $\Si_\cP\sim(A,B,C,D)$ with state space $\cX=\bigoplus_{j\in P}\cX_j$ and a subspace $\wtil{\cX}=\bigoplus_{j\in P}\wtil{\cX}_j$ such that $\wtil{\cX}_j\subseteq\cX_j$ for each $j$. If $\wtil{A},\wtil{B},\wtil{C}$ are the compressions of $A,B,C$ to $\wtil{\cX}$, respectively, and
\[
CA^kB = \wtil{C}\wtil{A}^k\wtil{B} \quad \text{for }k=0,1,\ldots,
\]
then the realization $\wtil{\Si}_\cP\sim(\wtil{A},\wtil{B},\wtil{C},D)$ is called a \emph{poset-causal reduction} of $\Si_\cP$ to the subspace $\wtil{\cX}$.
\end{defn}

Using Lemma \ref{L:EquivReal} and the state space subspaces that underlie our notions of controllability and observability defined in Sections \ref{S:UpControl} and \ref{S:DownObserve}, we obtain the following candidate for a poset-causal reduction.

\begin{prop}\label{P:KalDecStruc}
Consider a poset-causal system $\Si_{\cP}\sim(A,B,C,D)$ with associated subspaces $\wtil{\cR}$, $\overline{\cR}$ and $\overline{\cN}$ as defined in \eqref{eqR_joverlinetilde} and \eqref{eqN_joverlinetilde}. Define the subspace $\wtil{\cX}=\wtil{\cR}\ominus(\overline{\cR}\cap\overline{\cN})$ and let $\wtil{A}$, $\wtil{B}$ and $\wtil{C}$ be the compressions of $A$, $B$ and $C$ to $\wtil{\cX}$. Then $\wtil{\Si}_\cP\sim(\wtil{A},\wtil{B},\wtil{C},D)$ is a poset-causal reduction of $\Si_\cP$. Furthermore, if $P_{\cX_j}\wtil{\cX} =P_{\cX_j}\cX_{co}$ for all $j\in P$, and $\what{\cX}=\bigoplus_{j\in P} \what{\cX}_j$ is a subspace of $\cX$ with $\what{\cX}_j\subseteq \cX_j$ for all $j\in P$ so that $\cX_{co}\subseteq \what{\cX}$, then $\wtil{\cX}\subseteq \what{\cX}$.
\end{prop}

\begin{proof}[\bf{Proof}]
Using definitions \eqref{eqR_joverlinetilde} and \eqref{eqN_joverlinetilde}, it follows that
\begin{align*}
\wtil{\cX} =\bigoplus_{j=1}^p\wtil{\cX}_j,\quad \text{with}\ \ \wtil{\cX}_j = \wtil{\cR}_j\ominus(\overline{\cR}_j\cap\overline{\cN}^j)\subset \cX_j,
\end{align*}
where $\wtil{\cX}_j$, $\overline{\cR}_j$ and $\overline{\cN}^j$ are also as defined in \eqref{eqR_joverlinetilde} and \eqref{eqN_joverlinetilde}. Hence $\wtil{\cX}$ is a structured subspace of the state space. Since $\ov{\cR}\subseteq \cR\subseteq\wtil{\cR}$ and $\ov{\cN}\subseteq \cN$ by \eqref{eqContrIncl} and \eqref{eqObsIncl}, respectively, it follows from Lemma \ref{L:EquivReal} that $	CA^kB = \wtil{C}\wtil{A}^k\wtil{B}$ for $k=0,1,\ldots$ and hence $\wtil{\Si}_\cP$ is a poset-causal reduction of $\Si_\cP$.

For the final claim, assume $\wtil{\cX}_j = P_{\cX_j}\wtil{\cX}=P_{\cX_j}\cX_{co}$ for all $j\in P$ and let $\what{\cX}$ be as in the proposition, then
\[
\wtil{\cX}_j = P_{\cX_j}\cX_{co} \subseteq P_{\cX_j}\what{\cX} = \what{\cX}_j, \quad j\in P. \qedhere
\]
\end{proof}

In the above proposition we worked with the subspaces $\wtil{\cR}$, $\overline{\cR}$ and $\overline{\cN}$ since they have a natural interpretation in the context of the poset-causal system $\Si_\cP$ and satisfy the inclusion conditions of Lemma \ref{L:EquivReal}. Furthermore, $\wtil{\cR}$ is the smallest structured subspace of $\cX$ that contains $\cR$ and $\overline{\cN}$ is the largest structured subspace of $\cX$ contained in $\cN$, but $\ov{\cR}$ need not be the largest structured subspace of $\cX$ contained in $\cR$, unless if $\ov{\cR}=\cR^\circ$. A potentially smaller structured subspace that contains $\cX_{co}$ is thus given by $\wtil{\cX}':=\wtil{\cR} \ominus (\cR^\circ\cap\overline{\cN})$. However, despite the fact that $\cR^\circ$ is the largest structured subspace contained in $\cR$, it need not be the case that
$P_{\cX_j}\wtil{\cX}'=P_{\cX_j}\cX_{co}$ for all $j\in P$, as illustrated in the next example.

\begin{ex}\label{E:NonOpt}
	Consider a leader-follower system $\Si_\cP\sim(A,B,C,D)$, where $\cP=(P,\preceq)$ is the poset with $P=\{1,2\}$ and $1\preceq2$. Suppose $A \in \cI_\cP^{\un{n}\times\un{n}}$, $B \in \cI_\cP^{\un{n}\times\un{m}}$ and $A \in \cI_\cP^{\un{r}\times\un{n}}$ with $\un{n}=(2,2)$, $\un{m}=(2,1)$ and $\un{r}=(1,1)$, are given by
\begin{align*}
A = \left[\begin{array}{c c | c c}
1 & 0 & 0 & 0\\
0 & 0 & 0 & 0\\\hline
0 & 0 & 0 & 0\\
0 & 0 & 1 & 0
\end{array}\right], \quad
B = \left[\begin{array}{c c |c}
1 & 0 & 0\\
0 & 1 & 0 \\ \hline
0 & 0 & 0\\
0 & 1 & 0
\end{array}\right] \ands
C = \left[\begin{array}{c c|c c}
1 & 0 & 0 & 0\\\hline
1 & 0 & 1 & 0
\end{array}\right].
\end{align*}
Then $\cX_1=\text{span}\{e_1,e_2\}$ and $\cX_2=\text{span}\{e_3,e_4\}$. We can calculate the reachable space $\cR$ and the unobservable space $\cN$ using \eqref{eqreachunobs} and \eqref{eqcCcO}:
\[
\cR = \text{span}\{e_1,e_2+e_4\} \ands \cN = \text{span}\{e_2,e_4\}.
\]
In this case, we have $\cX_{co}=\cR\ominus(\cR\cap\cN)=\text{span}\{e_1\}$ and hence
\[
P_{\cX_1}\cX_{co} = \text{span}\{e_1\} \ands P_{\cX_2}\cX_{co} = \{0\}.
\]
Now we consider the structured space $\wtil{\cX}' = \wtil{\cR}\ominus(\cR^\circ\cap\overline{\cN}) = \bigoplus_{j=1,2}\wtil{\cX}_j'$. Then $\cX_j' = \wtil{\cX}_j\ominus (\cR^\circ_j \cap \overline{\cN}^j)$ and we can compute $\wtil{R}_j$, $\cR^\circ_j$ and $\overline{\cN}^j$ for $j=1,2$ using Lemma \ref{L:Rtilde} and Theorem \ref{T:Obs}:
\begin{align*}
\wtil{\cR}_1 & = P_{\cX_1}\cR = \text{span}\{e_1,e_2\}, &
\cR^\circ_1 &= \cX_1\cap\cR= \{e_1\}, &
\overline{\cN}^1 & = \cX_1\cap\cN = \text{span}\{e_2\},\\
\wtil{\cR}_2 & = P_{\cX_2}\cR = \text{span}\{e_4\}, &
\cR^\circ_2 &= \cX_2\cap\cR= \{0\}, &
\overline{\cN}^2 & = \cX_2\cap\cN = \text{span}\{e_4\},
\end{align*}
which gives
\[
\cX_1' = \wtil{\cR}_1 \ominus (\cR^\circ_1 \cap \overline{\cN}^1) = \cX_1 \ands
\cX_2' = \wtil{\cR}_2 \ominus (\cR^\circ_2 \cap \overline{\cN}^2) = \text{span}\{e_4\}.
\]
For both $j=1,2$, we see that $P_{\cX_j}\cX_{co} \neq \cX'_j$. Hence, despite $\cX_{co}$ being a structured subspace of $\cX$, of dimension 1, our approximation obtained from Proposition \ref{P:KalDecStruc} is a structured subspace of dimension 3. One can further check that in this case $\ov{\cR}=\cR^\circ$.
\end{ex}

There are many different choices of state space subspaces to compress the matrices $A$, $B$ and $C$ to a minimal realisation, which may or may not be structured, and when it is not structured, there may or may not be a natural way to embed this subspace in a structured subspace of $\cX$ for which compressed matrices preserve the moments. In this paper we have chosen to work with the space $\cX_{co}=\cR\ominus(\cR\cap\cN)$, since it appears naturally in the Kalman decomposition of the system and there are natural structured analogues of the observability and controllability spaces that meet the requirements. Alternatively, using a duality argument, one can also work with $\cN^\perp$ and $\cR^\perp$ instead of $\cR$ and $\cN$, respectively. In this case, the (possibly) non-structured subspace becomes $\cN^\perp \ominus (\cN^\perp \cap \cR^\perp)$, which can be embedded in the structured subspace $\wtil{\cN}^\perp \ominus (\ov{\cN}^\perp \cap \wtil{\cR}^\perp)$, or in the (possibly) smaller structured subspace $\cN^{\circ \perp} \ominus (\ov{\cN}^\perp \cap \wtil{\cR}^\perp)$. In the above example, in fact, it turns out that all three subspaces of $\cX$ are the same, so that in this case it is better to work with $\cN^\perp$ and $\cR^\perp$ instead of $\cR$ and $\cN$.

\begin{ex}
Let $A$, $B$ and $C$ as well as $\cP$ be as in Example \ref{E:NonOpt}. In this case one can compute that $\wtil{\cN}=\cN^\circ=\cN=\text{span}\{e_2,e_4\}$, while it was already observed that $\ov{\cN}=\cN$, $\cR=\text{span}\{e_1, e_2+e_4\}$ and $\wtil{\cR}=\text{span}\{e_1,e_2,e_4\}$. Hence
\[
\ov{\cN}^\perp\!\! =\wtil{\cN}^\perp\!\! =\cN^{\circ \perp}\!\! =\cN^\perp\!\! =\text{span}\{e_1,e_3\},\  \
\cR^\perp\!\! =\text{span}\{e_2-e_4,e_3\} ,\ \
\wtil{\cR}^\perp\!\! = \text{span}\{e_3\}.
\]
From this we obtain that
\[
\cN^\perp \ominus (\cN^\perp \cap \cR^\perp)=\text{span}\{e_1\}=\wtil{\cN}^\perp \ominus (\ov{\cN}^\perp \cap \wtil{\cR}^\perp).
\]
Hence, in this case, when compressing to $\wtil{\cN}^\perp \ominus (\ov{\cN}^\perp \cap \wtil{\cR}^\perp)$ a minimal realisation is obtained.
\end{ex}

\section{Conclusion}

In this paper we initiated a study into the realization theory for the class of poset-causal systems defined by Shah and Parrilo \cite{SP08}. Various notions of controllability and observability for poset-causal systems were introduced and investigated and their relations under duality were explained. The additional structure of a poset-causal system includes a decomposition of the state, input and output space of the global system as an orthogonal sum of the state, input and output spaces of the subsystems. While the classical notions of controllability and observability are based on the spaces of reachable and indistinguishable states, and as a result, in general, will not respect the poset-causal structure, we introduced notions based on reachable and indistinguishable states from the perspective of the subsystems, using upstream and downstream systems, which led to subspaces of reachable and indistinguishable states that do respect the poset-causal system structure, in the sense that they decompose as orthogonal sums of subspaces of the local state spaces of the subsystems, so-called structured subspaces. For some of these new notions, it turned out that they are optimal, in the sense that there is no better structured subspace of the state space to approximate the space of reachable or indistinguishable states, either as an upper or lower bound.

Using the new notions of reachable and indistinguishable states, we presented a variation of the Kalman reduction formula, which can be used to determine a poset-causal system of reduced size that still preserves the input-output map of the original system. Despite the optimality results for some of the new spaces of reachable and indistinguishable states on which the Kalman reduction is based, it turns out that the Kalman reduction obtained in this paper is not necessarily of minimal size. Hence, the question of how to determine a minimal poset-causal system that preserves the input-output map of a given poset-causal system remains a topic of further study.\smallskip

\paragraph{\bf Acknowledgments}
This work is based on research supported in part by the National Research Foundation of South Africa (NRF) and the DSI-NRF Centre of Excellence in Mathematical and Statistical Sciences (CoE-MaSS). Any opinion, finding and conclusion or recommendation expressed in this material is that of the authors and the NRF and CoE-MaSS do not accept any liability in this regard.

\end{document}